 \documentclass[11pt]{article}
\usepackage{amsmath,amsthm,amsfonts,epsfig}
\usepackage[affil-it]{authblk}
\usepackage{amssymb}
\usepackage{enumerate}
\usepackage[mathscr]{euscript}
\usepackage{tikz}
\usepackage{amssymb}
\usepackage{amsmath}
\usepackage{amsthm}
\usepackage{mathrsfs}
\usepackage{geometry}

\DeclareMathOperator{\diam}{diam}
\DeclareMathOperator{\dist}{dist}
\DeclareMathOperator{\dis}{dis}

\newtheorem{thm}{Theorem}[section]
\newtheorem{lemma}[thm]{Lemma}
\newtheorem{cor}[thm]{Corollary}
\newtheorem{prop}[thm]{Proposition}
\theoremstyle{definition}
\newtheorem{defn}[thm]{Definition}

\newtheorem{rmk}[thm]{Remark}

\begin{document}
\title{Gromov-Hausdorff distance with boundary and its application to Gromov hyperbolic spaces and uniform spaces}

\author{Hyogo Shibahara%
\thanks{Email: \texttt{shibahho@mail.uc.edu}}}
\affil{Department of Mathematical Sciences, University of Cincinnati,\\
P.O. Box 210025, Cincinnati, OH 45221-0025, U.S.A.}
\date{}
%\title{Gromov-Hausdorff distance with boundary and its application to Gromov hyperbolic spaces and uniform spaces}
%\author{Hyogo Shibahara}
%\thanks{Electronic address: \texttt{yfarjoun@math.mit.edu}}
%\affil{G. Mill\'an Institute of Fluid Dynamics,\\ Nanoscience and Industrial Mathematics,\\ Universidad Carlos III de Madrid, Spain}
%\affiliation{Department of Mathematics, University of Cincinnati, P.O. Box 210025, Cincinnati, OH 45221-0025, U.S.A}

\maketitle

\begin{abstract}
In this paper we introduce a notion of the Gromov-Hausdorff distance with boundary, denoted by $d_{GHB}$, to construct a framework of convergence of noncomplete metric spaces. We show that a class of bounded $A$-uniform spaces with diameter bounded from below is a complete metric space with respect to $d_{GHB}$. As an application we show the stability of Gromov hyperbolicity, roughly starlike property, uniformization, quasihyperbolization, and boundary of Gromov hyperbolic spaces under appropriate notions of convergence and assumptions. 
\end{abstract}


\section{Introduction}
\allowdisplaybreaks{
 The notion of uniform domains was first introduced in \cite{MarSar}. Since then, they have played a key role in deriving many important results in the context of quasiconformal mappings, such as extendability of Sobolev functions. Although a uniform domain was classically considered to be in the Euclidean plane, recent studies on uniform domains in metric measure spaces with doubling measure supporting $p$-Poincar\'e inequality have been also conducted, see, for example, \cite{A}, \cite{BHK}, \cite{BBS}, \cite{BSh}, \cite{GeO},\cite{Mar}, \cite{Jussi2}). Also, in \cite{BHK}, uniform domains were generalized as noncomplete locally compact metric spaces that satisfy the criteria of Definition~\ref{uniform space}, which we call  uniform spaces. Moreover, in \cite{BHK} it has been shown that there is a connection between proper geodesic roughly starlike Gromov hyperbolic spaces and uniform spaces and that if a  roughly starlike Gromov hyperbolic space $(X, d)$ is uniformized with a sufficiently small parameter $\epsilon>0$, then the resulting space, denoted by $(X^{\epsilon}, d_{\epsilon})$, is a uniform space. Conversely, a uniform space $\Omega$ with the quasihyperbolic metric $k$ is a Gromov hyperbolic space. Definitions of these terms are introduced in the next section. The above connection can be seen as the generalization of the correspondence between the Poincar\'e disk model and the Euclidean disk. Quite recently Butler obtained another uniformization result, which is motivated by the correspondence between the upper half plane and the hyperbolic space, see \cite{Bu}.  \\
\indent The Gromov-Hausdorff distance was originally defined in \cite{Ed}. It has been widely used in Riemannian geometry and analysis on metric spaces. Since every uniform space $\Omega$ is a noncomplete locally compact metric space with additional properties as in Definition~\ref{uniform space}, we need to consider the boundary $\partial \Omega := \overline{\Omega}\setminus \Omega$, in addition to the  metric space $\Omega$ itself. There are a few studies on convergence of Riemannian manifolds and metric spaces with boundary, see \cite{PS} and \cite{Raquel}. In \cite{PS}, a new notion of limit spaces called glued limit spaces was introduced for a sequence of noncomplete precompact metric spaces through $\delta$-inner regions. It was observed there that in some situations there is no Gromov-Hausdorff limit but a glued limit space still exists. Another feature is that this limit is realized as a subset of the Gromov-Hausdorff limit  if, of course, the Gromov-Hausdorff limit exists. Glued limit spaces, however, do not take care of the issue of the limit of the sequence of boundaries very well in general even if there is a Gromov-Hausdorff limit for a sequence of metric spaces, when we take the completion of a glued limit space. On the other hand, as we mentioned above, uniform spaces have to be treated together with their boundaries. It is therefore natural to ask how we could extend the notion of the Gromov-Hausdorff distance, taking  boundary into consideration. One  possible way to do so is a simple modification of the Gromov-Hausdorff distance. For metric spaces $X$ and $Y$, we consider
\begin{equation}
 	d_{GHB}(X, Y) := \text{inf} \ d^{Z}_{H}(\varphi(X), \psi(Y))+d^{Z}_{H}(\varphi(\partial X), \psi(\partial Y))
  \notag
 \end{equation}
 where $d^{Z}_{H}$ is the Hausdorff distance in $Z$ and infimum is taken over all metric spaces $Z$ and isometric embeddings $\varphi : \overline{X} \to Z$ and $\psi : \overline{Y} \to Z$.  Note that $\partial X := \overline{X}\setminus X$ where $\overline{X}$ is a metric completion of $X$. This definition allows us to distinguish between a metric space and its completion while the Gromov-Hausdorff distance does not see the difference between them.
We adopt $d_{GHB}$ as our definition of the Gromov-Hausdorff distance with boundary and we will show basic properties of $d_{GHB}$ for the space of all noncomplete precompact locally compact metric spaces.  \\
\indent In \cite{HRS} various types of convergence of a sequence of sets with quasihyperbolic distances, such as the Gromov-Hausdorff convergence and the Carath\'eodory convergence, were investigated. Since a notion of Gromov-Hausdorff distance with boundary has not been established, the paper \cite{HRS} deals with a sequence of sets in a same metric space. So one could ask how the results in \cite{HRS} would be extended for a sequence of noncomplete metric spaces, see Theorem~\ref{fouth theorem} for a related result.\\
\indent The goal of this paper is to examine the metric $d_{GHB}$ and properties that a limit space has from the geometric function theoretic viewpoint. Hence, the following questions are addressed in this paper:
\begin{enumerate}
\item  Is $d_{GHB}$ a metric? Is the class of all bounded uniform spaces complete with respect to this metric?
\item Are Gromov hyperbolicity and roughly starlike property stable under the pointed Gromov-Hausdorff convergence?
	\item Are uniformized spaces and their boundaries stable under the pointed Gromov-Hausdorff convergence?
	\item  Is quasihyperbolization stable under the Gromov-Hausdorff convergence with boundary?
\end{enumerate}
In this paper, we provide positive answers for the first three questions. Here we list the theorems.
\begin{thm}\label{third theorem}
	Let $A>0$ and $R>0$. Let $\mathcal{M}$ be the class of all noncomplete precompact locally compact metric spaces and $\mathcal{U}(A, R)$ be the class of all bounded  $A$-uniform spaces with the diameter bounded from below by $R>0$. Then $(\mathcal{M}, d_{GHB})$ is a (noncomplete) metric space and $(\mathcal{U}(A, R), d_{GHB})$ is a complete metric space.
\end{thm}

%Remark that in the above statement, to be precise, we consider the set of all isometry  classes of metric spaces, see Section 3. 
We remark that there is one direct application to analysis on metric spaces. The fact that $\mathcal{U}(A, R)$ is complete with respect to $d_{GHB}$ together with \cite[Theorem 4.4]{BSh}, \cite[Theorem 3]{Kei} and \cite[Chapter 9]{Ch} gives us a stability of metric measure spaces with doubling measure supporting $p$-Poincar\'e inequality, which is new for a sequence of noncomplete metric measure spaces, see Remark \ref{PI space}.

\begin{thm}\label{first theorem}
	Let $(X_n, d_n, p_n)_n$ be a sequence of pointed proper geodesic $\delta$-Gromov hyperbolic spaces which is pointed Gromov-Hausdorff convergent to a complete metric space $(X, d, p)$. Then, $X$ is a proper geodesic $\delta$-Gromov hyperbolic space. Moreover, if the spaces $(X_n, d_n, p_n)_n$ are further assumed to be $M$-roughly starlike, then there exists $\tilde{M} := \tilde{M}(M, \delta)$ such that $X$ is $\tilde{M}$-roughly starlike.
\end{thm}

\begin{thm}\label{second theorem}
	Let $(X_n, d_n, p_n)_n$ be a sequence of pointed proper geodesic $M$-roughly starlike $\delta$-Gromov hyperbolic spaces which is pointed Gromov-Hausdorff convergent to a complete metric space $(X, d, p)$. For each $0< \epsilon \leq \epsilon_0(\delta)$, let $(X^{\epsilon}_n, d_{n, \epsilon})_n$ the sequence of uniformized spaces. Then $(X^{\epsilon}_n, d_{n, \epsilon})_n$ Gromov-Hausdorff converges to $(X^{\epsilon}, d_{\epsilon})$  with boundary. 
\end{thm}

\indent   %In order to prove the convergence of uniformized spaces $(X_n, d_{n, \epsilon})_n$, the key lemma is to prove the convergence of balls $(\bar{B}_{d_n}(p_n, R), d_{n, \epsilon})_n$. Even though under the setting of Theorem \ref{second theorem}, we priori know that $(\bar{B}_{d_n}(p_n, R), d_{n})_n$ Gromov-Hausdorff converges to $(\bar{B}_d(p, R), d)$ and a ball with the original metric is bilipschitz equivalent to the ball with the uniformization metric, the combination of the Gromov-Hausdorff convergence of $(\bar{B}_{d_n}(p_n, R), d_{n})_n$ and bilipschitz equivalence does not seem to directly derive the Gromov-Hasudorff convergence of $(\bar{B}_{d_n}(p_n, R), d_{n, \epsilon})_n$. 
%In fact, once can easily construct a Gromov-Hausdorff convergent sequence $(X_n, d_n)$ to $(X, d)$ while $(X_n, \tilde{d}_n)_n$ does not Gromov-Hausdorff converge to $(X, \tilde{d})$ where $(X_n, \tilde{d}_n)$ and $(X, \tilde{d})$ are bilipschitz equivalent to $(X_n, d_n)$ and $(X, d)$, respectively.  Hence we discretize curves to estimate the uniformization metric and take the advantage of having the same metric deformation for all spaces. 
We remark that the stability of the Gromov hyperbolicity holds true, under the assumption that the limit exists for a given sequence. This assumption is valid if one considers a sequence of intrinsic $A$-uniform spaces endowed with quasihyperbolic metric where diameter of spaces and doubling constants with respect to original metrics are uniformly bounded, see Theorem \ref{fouth theorem} and Theorem \ref{Compactness for uniform spaces}. On the other hand, in general a pointed Gromov-Hausdorff limit may not exist for some sequence of Gromov hyperbolic spaces. \\
\indent It is clear that for given noncomplete precompact locally compact metric spaces $X$ and $Y$, we have 
\begin{equation}\label{converse not}
	d_{GH}(\overline{X}, \overline{Y})+ d_{GH}(\partial X, \partial Y) \leq d_{GHB}(X, Y).
\end{equation}
 Hence Theorem~\ref{second theorem} also states that $(\overline{X^{\epsilon}_n}, d_{n, \epsilon})_n$ and $(\partial_{d_{n, \epsilon}} X_n^{\epsilon}, d_{n, \epsilon})_n$ Gromov-Hausdorff converge to $(\overline{X^{\epsilon}}, d_{\epsilon})$ and $(\partial_{d_{\epsilon}} X^{\epsilon}, d_{\epsilon})$, respectively. On the other hand, the converse inequality of (\ref{converse not}) does not hold, see Remark~\ref{d_{GH} versus d_{GHB}} where we provide a simple example which also suggests that the sum in the left-hand side of \eqref{converse not} does not induce a metric. Also, by \cite[Proposition 4.13]{BHK}, for a given proper geodesic $M$-roughly starlike $\delta$-Gromov hyperbolic space $(X, d, p)$, the canonical bijective identification between the metric boundary of $X^{\epsilon}$ with respect to uniformization metric $d_{\epsilon}$, $(\partial_{d_{\epsilon}}X^{\epsilon}, d_{\epsilon})$, and the Gromov boundary of $X$ with the visual metric, $(\partial_{G}X, d_{p, \epsilon})$ are bilipschitz equivalent. Hence, Theorem~\ref{second theorem} can be seen as follows : If we change the visual metric $d_{p, \epsilon}$ of Gromov boundary to the metric $d_{\epsilon}$, which is  bilipschitz equivalent to $d_{p, \epsilon}$, then the sequence of Gromov boundaries $(\partial_G X_n, d_{n, \epsilon})$ converges to $(\partial_G X, d_{\epsilon})$. \\
\indent On the other hand, regarding the fourth question, in general the desired conclusion is false. Instead, we prove that if a sequence of uniform spaces $(\Omega_n, d_n)_n$ are all length spaces, then the quasihyperbolization is stable. 
%From Example in this paper, we can see that in general, we cannot expect that the quasihyperbolization is stable if we do not have the assumption of spaces to be length. 
To make the meaning precise, we state the last theorem.
\begin{thm}\label{fouth theorem}
	Let $(\Omega_n, d_n)_n \subseteq \mathcal{U}(A, R)$ and $(\Omega, d) \in \mathcal{M}$. Suppose $(\Omega_n, d_n)_n$ are all length spaces and $(\Omega_n, d_n)_n$ converges to $(\Omega, d)$ with respect to the Gromov-Hausdorff distance with boundary.  Then $(\Omega_n, k_n, p_n)_n$ is pointed Gromov-Hausdorff convergent to $(\Omega, k, p)$ for some $p_n \in \Omega_n$ and $p \in \Omega$, where $k_n$ and $k$ are quasihyperbolic metrics defined on $\Omega_n$ and $\Omega$, respectively.
\end{thm}

\indent The rest of this paper consists of six parts. In Section 2, we review basic concepts and propositions needed to state our results. In Section 3, we define the Gromov-Hausdorff distance with boundary $d_{GHB}$ on $\mathcal{M}$ and prove the first assertion of Theorem~\ref{third theorem}. We also explore some basic properties of $d_{GHB}$ that the Gromov-Hausdorff distance also has. In Section 4, we prove the second assertion of Theorem~\ref{third theorem}. A compactness theorem for $\mathcal{U}(A, R)$ is also provided. In Section 5 we prove Theorem~\ref{first theorem}. In section 6, we prove Theorem~\ref{second theorem}.  In Section 7, we give a proof of Theorem~\ref{fouth theorem}. 

\subsubsection*{\hfil Acknowledgement \hfil}
The author would like to thank Nageswari Shanmugalingam and Jeff  Lindquist for countless helpful discussions and  for carefully reading the manuscript. He also would like to appreciate his former advisor Atsushi Katsuda who brought him to this research area and gave some important ideas, remarks on negatively curved spaces and an alternative proof of Theorem \ref{second theorem}, which were invaluable. The author's research was partially supported by the grant DMS $\# 1800161$ from NSF (U.S.A.).

\section{Preliminaries}
We fix some notations. Let $(X, d)$ be a metric space. For $x \in X$ and $r>0$, we set $B_d(x,r):=\{y\in X\, |\, d(x,y)<r\}$ and $\bar{B}_d(x,r):=\{y\in X\, |\, d(x,y)\leq r\}$. The length of a curve $\gamma$ with respect to $d$ is denoted by $l_d(\gamma)$. The minimum and the maximum of $n$ real numbers $(a_i)_{i=1}^{n}$ is denoted by $a_1\wedge \cdots \wedge a_n$ and $a_1 \vee \cdots \vee a_n$, respectively. 
We first review the Gromov-Hausdorff distance and the pointed Gromov-Hausdorff convergence.
\begin{defn}($\epsilon$-neighborhood)
	 Let $\epsilon \geq 0$ be given and $X$ be a metric space. An $\epsilon$-neighborhood of a given set $A \subseteq X$, denoted by $A_{\epsilon}$, is defined by
	 \[
	 A_{\epsilon}:=\{x \in X\  |\  \text{There exists}\  y \in A \ \text{such that}  \ d(x, y)\leq \epsilon. \}.
	 \] 
	 Note that if $\epsilon=0$, then $A_{0}=A$ for every set $A \subseteq X$.
\end{defn}
\begin{defn}(Hausdorff distance)
	 Let $A, B \subseteq X$ be given. The Hausdorff distance, denoted by $d_{H}(A, B)$, is defined by
	 \[
	 d_{H}(A, B):=\inf \{r \geq 0  \  |\  A \subseteq B_r \  \text{and}  \  B \subseteq A_r \}.
	 \] 
	 When we want to emphasize which space or metric is used, we sometimes write $d^{X}_{H}$ or $d^{d}_{H}$.
\end{defn}
Using the Hausdorff distance, we define the Gromov-Hausdorff distance.
\begin{defn}(Gromov-Hausdorff distance)
	 Let $X, Y$ be bounded metric spaces. The Gomov-Hausdorff distance, denoted by $d_{GH}(X, Y)$, is defined by
	 \begin{equation}
	 d_{GH}(X, Y):=\inf \ d^{Z}_{H}(\varphi(X), \psi(Y))
  \notag
 \end{equation}
 where infimum is taken over all metric spaces $Z$ and isometric embeddings $\varphi : X \to Z$ and $\psi : Y \to Z$. 
\end{defn}
	We remark that $d_{GH}(X, \overline{X})=0$ where $\overline{X}$ is the completion of $X$. It is also well-known that for compact metric spaces $X$ and $Y$, $d_{GH}(X, Y)=0$ if and only if there exists an isometry from $X$ to $Y$. Moreover, the Gromov-Hausdorff distance $d_{GH}$ is a metric on the set of all isometry classes of compact metric spaces, see \cite{HKST} and \cite{BBI}.

\begin{defn}
Let $(X, d)$ be a metric space. Define 
	\begin{equation}
 	Net_{\epsilon}(X) := \inf \left\lbrace |M| \;\middle|\;
  \begin{tabular}{@{}l@{}}
   $M$ is an $\epsilon$-net in $X$
   \end{tabular}
 \right\rbrace \notag 
 \end{equation}
 and 
 \begin{equation}
 	Sep_{\epsilon}(X) := \sup \left\lbrace |M| \;\middle|\;
  \begin{tabular}{@{}l@{}}
   $M$ is $\epsilon$-separated in $X$
   \end{tabular}
 \right\rbrace. \notag
 \end{equation}
 where $|\cdot|$ is the cardinality of a set. Here we say that $M$ is an $\epsilon$-net in $X$ if $X \subseteq M_{\epsilon}$ and $M$ is $\epsilon$-separated in $X$ if $d(x, y) \geq \epsilon$ for every $x, y \in M$ with $x \neq y$.
\end{defn} 
The following two useful propositions are due to Gromov, see \cite[p.64-65]{Gr}.
\begin{prop}\label{cap-cov}
	Let $\mathcal{C}$ be a subset of the set of all compact metric spaces. Then, TFAE : 
	 \begin{enumerate}
	 	\item There exist $D > 0$ and $N : (0, \infty) \to \mathbb{N}$ such that for every $X \in \mathcal{C}$, $\diam(X)\leq D$ and for every $\epsilon>0$, $Net_{\epsilon}(X) \leq N(\epsilon)$.
	 	\item There exist $D > 0$ and $N : (0, \infty) \to \mathbb{N}$ such that for every $X \in \mathcal{C}$, $\diam(X)\leq D$ and for every $\epsilon>0$, $Sep_{\epsilon}(X) \leq N(\epsilon)$.
	 	\item $\mathcal{C}$ is totally bounded with respect to $d_{GH}$.
	 \end{enumerate}
\end{prop}
\begin{prop}\label{Gromov embedding}
	Let $\mathcal{C}$ be a subset of the set of all compact metric spaces. If $\mathcal{C}$ satisfies any of the conditions in Proposition~\ref{cap-cov}, then there is a compact set $K \subseteq l^{\infty}$ such that every $X \in \mathcal{C}$ admits an isometric embedding into $K$.
\end{prop}

Next, we define the pointed Gromov-Hausdorff convergence. Although some equivalent notions can be found in many papers, we follow the pointed Gromov-Hausdorff convergence in \cite[Definition 8.1.1]{BBI}  and  \cite[Definition 11.3.1]{HKST}. One can also check \cite{H} for other equivalent conditions and detailed explanation on the pointed Gromov-Hausdorff topology.
\begin{defn}(Pointed Gromov-Hausdorff convergence) \label{p-GH}
We say that \emph{$(X_n, d_n, p_n)_n$ is pointed Gromov-Hausdorff convergent to $ (X, d, p)$} if for every $r>0$ and $0<\epsilon<r$, there exists $N \in \mathbb{N}$ such that for each $n\geq N$, there exists $f^{\epsilon}_{n}: B_{d_n}(p_{n}, r) \to X$ such that 
\begin{enumerate}
	\item $f^{\epsilon}_{n}(p_{n})=p$,
	\item $|d(f^{\epsilon}_{n}(x), f^{\epsilon}_{n}(y))-d_n(x,y)|<\epsilon$ for all $x, y \in B_{d_n}(p_{n}, r)$,
	\item $B_{d}(p, r-\epsilon) \subseteq f^{\epsilon}_{n}(B_{d_n}(p_{n},r))_{\epsilon}$.
\end{enumerate}
%where $f^{\epsilon}_{n}(B_{d_n}(p_{n},r))_{\epsilon}$ is an $\epsilon$-neighbourhood of $f^{\epsilon}_{n}(B_{d_n}(p_{n},r))$.
\end{defn}
\begin{rmk}\label{p-GH def remark}
	A metric space $(X, d)$ is said be a length space if for every pair of points $x, y \in X$, the infimum of the lengths of curves from $x$ to $y$  coincides with $d(x, y)$. In the above definition, if the pointed metric spaces $(X_n, d_n, p_n)_n$ are all length spaces, then we can further assume that $f^{\epsilon}_{n}(B_{d_n}(p_{n}, r)) \subseteq B_{d}(p, r)$ for every $r>0$ and $0<\epsilon<r$. We leave it to the reader to verify this.
\end{rmk}

We set the terminology to state the compactness theorem for pointed metric spaces. A metric space $(X, d)$ is said to be proper if every bounded closed subset is compact. If each pair of points in $(X, d)$ is joined by a curve whose length equals the distance between them, we say that $(X, d)$ is a geodesic metric space.
\begin{defn}
	A sequence of pointed metric spaces $(X_n, d_n, p_n)_n$ is said to be \emph{pointed uniformly totally bounded} if there is a function $N : (0, \infty) \times (0, \infty) \to (0, \infty)$ such that, for each $0< \epsilon < R$ and $n \in \mathbb{N}$, the closed ball $\bar{B}_{d_n}(p_n, R)$ in $X_n$ contains an $\epsilon$-net of cardinality at most $N(\epsilon, R)$.
\end{defn}

The compactness theorem for pointed proper metric spaces first appeared in \cite[p.64]{Gr}. We state the compactness theorem under the pointed Gromov-Hausdorff convergence based on \cite[Theorem 11.3.16]{HKST}.
\begin{prop}\label{compactness}
Let $(X_n, d_n, p_n)_{n}$ be a sequence of pointed proper metric spaces. Suppose $(X_n, d_n, p_n)_{n}$ is pointed uniformly totally bounded. Then there exists a subsequence $(X_{n_{k}},d_{n_k}, p_{n_{k}})_{k}$ which is pointed Gromov-Hausdorff convergent to a complete metric space $(X, d, p)$.
\end{prop}
\begin{rmk}\label{limit is fine}
	If pointed metric spaces $(X_n, d_n, p_n)_{n}$ are all assumed to be proper and geodesic, then the limit pointed complete metric space $(X, d, p)$ is also proper and geodesic by the Hopf-Rinow theorem, see \cite[Proposition 11.3.12 and Proposition 11.3.14]{HKST} and \cite[Theorem 7.5.1]{BBI} for further discussion on the pointed Gromov-Hausdorff convergence.
\end{rmk}

\begin{defn}[Gromov product]
	Let $X$ be a metric space. The \emph{Gromov product} of two points $x, y \in X$ with respect to $m \in X$ is defined by
	\[
	 (x|y)_{m}:= \frac{1}{2}(d(m, x)+d(m, y)-d(x, y)).
	\]
\end{defn}
 
 In order to use uniformization results from \cite{BHK}, we assume $\delta$-Gromov hyperbolic spaces to be geodesic even though it is not required in the following four-point condition. Note that there are equivalent definitions of $\delta$-Gromov hyperbolicity, such as $\delta$-thinness of a geodesic triangle. See \cite{Bo}, \cite{BH}, \cite{GH}, and \cite{Jussi} for more discussion about Gromov hyperbolic spaces. 
\begin{defn}[Gromov hyperbolic space]
Let $(X, d)$ be a unbounded proper geodesic metric space. We say that $X$ is a \emph{$\delta$-Gromov hyperbolic space} if for all $x, y, z, m \in X$, 
\[
(x|z)_{m}\geq (x|y)_{m}\wedge(y|z)_{m}-\delta.
\]
\end{defn}

\begin{defn}($M$-Roughly starlike property)
	We say that a  pointed metric space $(X, d, p)$ is \emph{$M$-roughly starlike} if for every $x \in X$, there exists a geodesic ray $\gamma : [0, \infty) \to X$ with $\gamma(0)=p$ such that $\dist(x, \gamma):=\inf\limits_{t \in [0, \infty)}d(x, \gamma(t)) \leq M$.
\end{defn}

\begin{defn}(Quasiisometric embedding/path)
	Let $\mu \geq 1$ and $C>0$ be constants. We say that a map $f:X \to Y$ is a \emph{$(\mu, C)$-quasi-isometric embedding} if for every pair of points $x, y \in X$,
	\[
	\mu^{-1} d(x, y)-C\leq d(f(x), f(y))\leq \mu d(x, y)+C
	\] holds.
	In particular, if $X$ is an interval $[a, b]$, then $f : [a, b] \to Y$ satisfying the above condition is called a \emph{$(\mu, C)$-quasiisometric path}. Also, if a map $f : X \to Y$ is $(1, \epsilon)$-isometric embedding and $Y \subseteq (f(X))_{\epsilon}$, we call $f$ an \emph{$\epsilon$-isometry}.
\end{defn}
\begin{rmk}
	Given compact metric spaces $X$ and $Y$, the existence of an $\epsilon$-isometry $f : X \to Y$ implies that $d_{GH}(X, Y) \leq 3 \epsilon$. Conversely, if $d_{GH}(X, Y) < \epsilon$, then there exists a $5\epsilon$-isometry $f  : X \to Y$, see \cite[Proposition 7.4.11]{BBI} for the proof. We use this relation to prove the convergence of uniformized spaces and boundaries. 
\end{rmk}

%We give the stability result for $(\mu, C)$-quasiisomettric paths from \cite[Theorem 3.7]{Jussi}.
%\begin{prop}\label{geodesic stability}
	%Let $X$ be a $\delta$-Gromov hyperbolic space and $\varphi : [a, b] \to X$ and $\varphi' : [a', b'] \to X$ be $(\mu, C)$-quasiisometric paths with $\varphi(a)=\varphi'(a')$ and $\varphi(b)=\varphi'(b')$. Then there exists $M(\delta, \mu, C)$ such that 
	%\[
	%d_{H}(\varphi([a, b]), \varphi'([a', b'])) \leq M(\delta, \mu, C).
	%\]
	%where $d_{H}$ is the Hausdorff distance on $X$.
%\end{prop}

We next review the definitons of uniform spaces, uniformization, and quasihyperbolization. These definitions can also be found in \cite{BHK}.

\begin{defn}($A$-uniform curve)\label{uniform curve}
	Let  $A>0$ and a metric space $(\Omega, d)$ be given. We say that a curve $\gamma : [0, 1] \to \Omega$ is an \emph{$A$-uniform curve} if the curve $\gamma$ satisfies 
	\begin{enumerate}
		\item $l_{d}(\gamma) \leq A d(x, y)$,
		\item $l_{d}(\gamma|_{[0, t]})\wedge l_{d}(\gamma|_{[t, 1]}) \leq A$ dist$(\gamma(t), \partial \Omega)$ \ \ \ for every $t \in [0, 1]$.
	\end{enumerate}
\end{defn}

\begin{rmk}
	This definition does not depend on the parametrization of a curve $\gamma$.
\end{rmk}

\begin{defn}($A$-uniform space)\label{uniform space}
	 A locally compact, rectifiably connected noncomplete metric space $(\Omega, d)$ is called an \emph{$A$-uniform space} if every pair of points in $\Omega$ can be connected by an $A$-uniform curve.
\end{defn}
\begin{defn}(Uniformization)\label{uniformaization}
	Let $(X, d)$ be a $\delta$-Gromov hyperbolic space. Fix a base point $p \in X$ and $\epsilon>0$. The \emph{uniformization metric $d_{\epsilon}$} is defined by
	\[
	d_{\epsilon}(x, y):= \inf\limits_{\gamma}\int^{l_d(\gamma)}_{0}e^{-\epsilon d(p, \gamma(t))}\,dt
	\]
	where infimum is taken over all arc-length parametrized curves $\gamma$ from $x$ to $y$. The metric space $(X, d_{\epsilon})$ is called a \emph{uniformized space}. We denote this  uniformized space $(X, d_{\epsilon})$ by $(X^{\epsilon}, d_{\epsilon})$.  We note that a \emph{Harnack type inequality}
	\begin{equation}\label{Harnack}
		e^{-\epsilon d(x, y)} \leq \frac{e^{-\epsilon d(p,x)}}{e^{-\epsilon d(p,y)}} \leq e^{\epsilon d(x, y)}
	\end{equation} 
	holds for every $x, y \in X$, see \cite[Chapter 5]{BHK}
\end{defn}
\begin{rmk}\label{epsilon constraint}
	As mentioned in the previous section, it has been shown in \cite[Proposition 4.5 and Chapter 5]{BHK} that there exists $\epsilon_0(\delta)>0$ such that  for each $0<\epsilon \leq \epsilon_{0}(\delta)$, the uniformized space $(X^{\epsilon}, d_{\epsilon})$ is a uniform space. Throughout this paper, we fix the constant $\epsilon_{0}(\delta)$ and use it in the rest of this paper.
	\end{rmk}

\begin{defn}(Quasihyperbolizaton)
	Let $(\Omega, d)$ be an $A$-uniform space.  The \emph{quasihyperbolic metric $k$} is defined by
	\[
	k(x, y):= \inf\limits_{\gamma}\int^{l_d(\gamma)}_{0}\frac{1}{d(\gamma(t))}\,dt
	\]
	where infimum is taken over all arc-length parametrized curves $\gamma$ from $x$ to $y$, and $d(\cdot):= \dist( \cdot , \partial \Omega)$.  
\end{defn}
By \cite[Theorem 3.6]{BHK}, if $(\Omega, d)$ is a uniform space, then $(\Omega, k)$ is a proper geodesic $\delta$-Gromov hyperbolic space for some $\delta=\delta(A)$. Moreover, if $(\Omega, d)$ is bounded, then $(\Omega, k)$ is $M$-roughly starlike for some $M=M(A)$.

In order to prove the boundary convergence in the next section, we briefly  review the connection between the Gromov boundary and the metric boundary $(\partial_{d_{\epsilon}}X^{\epsilon}, d_{\epsilon})$. We refer the interested reader to \cite{BHK} for a detailed discussion and \cite{Jussi} and \cite{BH} for Gromov boundary.
\begin{defn}(Gromov boundary)
	We say that for a fixed point $p \in X$, a sequence $(x_{n})_{n}$ is a \emph{Gromov sequence} if $(x_{i}| x_{j})_{p}\to \infty$ as $i, j \to \infty$. Also we say two Gromov sequences $(x_{n})_{n}$ and $(y_{n})_{n}$ are equivalent if $(x_{n}|y_{n})_{p} \to \infty$ as $n \to \infty$. We write $(x_n)_n \sim (y_n)_n$ if they are equivalent. Since the space $X$ is $\delta$-Gromov hyperbolic, this is an equivalence relation. Hence we can consider the quotient space of the set of all Gromov sequences, denoted by $\partial_{G}X$. We call $\partial_{G}X$ the \emph{Gromov boundary}. 
\end{defn}
Let $(X, d)$ be a $\delta$-Gromov hyperbolic space and $p \in X$ be a fixed base point. The Gromov product is extended to $X\cup \partial_{G}X$. %by
%\[
%(a | b)_p :=\text{inf}\{ \liminf_{n \to \infty}(x_n | y_n)_p \ |\ (x_n)_n \in a \ \text{and} \ (y_n)_n \in b \}
%\]
%for given $a$, $b \in \partial_{G}X$.  
We refer the reader to \cite{BH} and \cite{Jussi2} for more discussion on extension of the definition of Gromov product to $X\cup \partial_{G}X$. 
\begin{defn}(The metric on $\partial_{G}X$)
	Let $(X, d)$ be a $\delta$-Gromov hyperbolic space. Let $p \in X$ be a fixed point and $0<\epsilon<1/(5\delta)$. For $x, y \in \partial_{G} X$, define 
	\begin{equation}
		\rho_{p, \epsilon}(x, y):= e^{-\epsilon(x|y)_{p}}.
	\end{equation} 
	We set the metric
	\begin{equation}
		d_{p, \epsilon}(x, y) := \inf \sum\limits_{i=1}^{n} \rho_{p, \epsilon}(a_{i-1}, a_i)
	\end{equation}
	where infimum is taken over all finite sequences $(a_i)_{i=0}^{n} \subseteq X\cup \partial_{G}X$ with $a_0=x$ and $a_n=y$.
\end{defn}

\begin{rmk}\label{Gromov identification}
 There is another characterization of the Gromov boundary using geodesic rays. We say that two geodesic rays $\gamma$ and $\tilde{\gamma}$ in $X$ with $\gamma(0)=\tilde{\gamma}(0)=p$ are equivalent if $\sup_{t \geq 0}d(\gamma(t), \tilde{\gamma}(t))$ is finite. Then we can consider a quotient space of a set of all geodesic rays emanating from $p \in X$ by the above equivalence relation, denoted by $\partial_{r}X$. By \cite[Lemma 3.13]{BH}, if a metric space $X$ is proper geodesic, then the map $\partial_{r}X \ni [\gamma] \mapsto [(\gamma(n))_n] \in \partial_{G}X$ is a bijective canonical identification between the Gromov boundary $\partial_{G}X$ and $\partial_{r}X$. Moreover, by \cite[Lemma 4.10, Theorem 4.13]{BHK}, we know that the map $\partial_{G}X \ni [(\gamma(n))_n] \mapsto \lim\limits_{n \to \infty} \gamma(n) \in \partial_{d_{\epsilon}}X^{\epsilon}$ is well-defined and is a canonical bilipschitz map for $0<\epsilon \leq \epsilon_{0}(\delta)$ where $\epsilon_{0}(\delta)$ is a constant defined in \cite[Chapter 5]{BHK} and $\partial_{d_{\epsilon}}X^{\epsilon}$ is the metric boundary defined by $\overline{X^{\epsilon}}\setminus X^{\epsilon}$. This fact plays a key role to prove the convergence of boundary.
% As a summary, the important consequence from these facts is that for each $x \in \partial_{d_{\epsilon}}X^{\epsilon}$, we can always find a geodesic ray $\gamma$ such that $\gamma(n) \to x$ with respect to $d_{\epsilon}$ as $n \to \infty$. 
\end{rmk}

\section{The notion of Gromov-Hausdorff distance with boundary}
In this section, we define the notion of Gromov-Hausdorff distance with boundary, denoted by $d_{GHB}$, and examine basic properties of $d_{GHB}$. We omit some proofs of the propositions provided below if there is no change in their proofs in comparison to the analogs of the Gromov-Hausdorff distance found in \cite{HKST} and \cite{BBI}.

\begin{defn}(Hausdorff distance with boundary)
	Let $(X, d)$ be a metric space and $\mathcal{M}(X)$ denote the set of all noncomplete precompact locally compact subsets of $X$. For $A, B \in \mathcal{M}(X)$, we set the \emph{Hausdorff distance with boundary} by
	\begin{equation}
 	d_{HB}(A, B) :=  d^{X}_{H}(A, B)+d^{X}_{H}(\partial A, \partial B)\notag 
 	  \end{equation}
 	where $\partial A:= \overline{A}\setminus A$ and $\partial B:= \overline{B}\setminus B$ are the boundaries of $A$ and $B$ respectively. One can check that this $d_{HB}$ is a metric on $\mathcal{M}(X)$.
\end{defn}

For a given metric space $X$, denote the metric boundary by $\partial X:= \overline{X} \setminus X$ where $\overline{X}$ is the completion of $X$. We now define the Gromov-Hausdorff distance with boundary.
\begin{defn}[Gromov-Hausdorff metric with boundary]
	 Let $X$ and $Y$ be noncomplete precompact metric spaces. Set
	\begin{equation}
 	d_{GHB}(X, Y) := \text{inf} \ d^{Z}_{H}(\varphi(X), \psi(Y))+d^{Z}_{H}(\varphi(\partial X), \psi(\partial Y))
  \notag
 \end{equation}
 where infimum is taken over all metric spaces $Z$ and isometric embeddings $\varphi : \overline{X} \to Z$ and $\psi : \overline{Y} \to Z$.
\end{defn}

%Note that if one of $\partial X$ and $\partial Y$ is empty, we interpret $d_{GHB}(X, Y)$ as $\infty$. Also if both $\partial X$ and $\partial Y$ are empty, we interpret $d^{Z}_{H}(\phi(\partial X), \phi(\partial Y))$ as $0$. Hence this quasimetric coincides with the Gromov-Hausdorff distance if both spaces $X$ and $Y$ are compact. 
In this section, unless otherwise stated, we always assume that metric spaces are all noncomplete and precompact. We first record the following elementary yet useful remarks.

\begin{rmk}\label{extension}
	If a map $f$ is an isometry from $X$ to $Y$, then both the canonical extension $\tilde{f} : \overline{X} \to \overline{Y}$ and the restriction $\tilde{f}|_{\partial X} : \partial X \to \partial Y$  are isometries.
\end{rmk}

\begin{rmk}\label{isometry1}
	If there exists an isometry $f : X \to Y$, then $d_{GHB}(X, Y)=0$.
\end{rmk}

\begin{defn}
	Let $X, Y$ be metric spaces. We say that $X \sqcup Y := (X\times \{0\})\cup (Y\times \{1\})$ is a \emph{disjoint union} of the sets $X$ and $Y$. The canonical maps $\iota_{X} : X \hookrightarrow X \sqcup Y$ and $\iota_{Y} : Y \hookrightarrow X \sqcup Y$ are defined by 
	\[
	\iota_{X}(x)=(x, 0) \ \ \text{and} \ \ \iota_{Y}(y)=(y, 1)
	\]
	for each $x \in X$ and $y \in Y$.
\end{defn}

\begin{defn}(Admissible metric)
	Let $X, Y$ be metric spaces. We say a metric $d$ on $X \sqcup Y$ is \emph{admissible} if the canonical maps $\iota_{X} : X \hookrightarrow X \sqcup Y$ and $\iota_{Y} : Y \hookrightarrow X \sqcup Y$ are isometric embeddings with respect to the metric $d$.
\end{defn}

We only give a sketch of the proof of the following lemma since the proof is identical with the case of the Gromov-Hausdorff distance. See \cite[Remark 7.3.12 and Proposition 7.3.16]{BBI} and \cite[Proposition 11.1.9]{HKST}.
\begin{lemma}\label{equivalence1}
	Let $X$ and $Y$ be noncomplete precompact metric spaces. Then 
	\[d_{GHB}(X, Y)=\inf\limits_{d}(d^{d}_{H}(\iota_{\overline{X}}(X), \iota_{\overline{Y}}(Y))+d^{d}_{H}(\iota_{\overline{X}}(\partial{X}), \iota_{\overline{Y}}(\partial{Y})))=: \tilde{d}_{GHB}(X, Y),
	\]
	where the infimum is taken over all admissible metrics on $\overline{X} \sqcup \overline{Y}$ and $d^{d}_{H}$ is the Hausdorff distance with respect to the metric $d$ on $\overline{X} \sqcup \overline{Y}$.
\end{lemma}
\begin{proof}
	Since $d_{GHB}(X, Y) \leq \tilde{d}_{GHB}(X, Y)$ is obvious, it suffices to show that $\tilde{d}_{GHB}(X, Y) \leq d_{GHB}(X, Y)$. For any $\epsilon>0$, choose a metric space $Z$ and the isometric embeddings $\varphi : \overline{X} \to Z$ and $\psi : \overline{Y} \to Z$ such that 
	\begin{equation}\label{claim1}
		d^{Z}_{H}(\varphi(X), \psi(Y))+d^{Z}_{H}(\varphi(\partial X), \psi(\partial Y))< d_{GHB}(X, Y)+ \epsilon
	\end{equation}
	By considering $Z':=Z\times [0, 1]$ with the product metric, $\varphi(\overline{X})':=\varphi(\overline{X})\times \{0\}$, and $\psi(\overline{Y})':=\psi(\overline{Y})\times \{\epsilon\}$ if needed, we may assume that $\varphi(\overline{X})$ and $\psi(\overline{Y})$ are disjoint.  Define $d$ on $\overline{X} \sqcup \overline{Y}$ by 
	\[
	d(x, y)  :=
     \begin{cases}
       d_{Z}(\varphi(x), \varphi(y)) &\quad\text{if $x, y \in \overline{X}$}\\
       \text{$d_{Z}(\varphi(x), \psi(y))$} &\quad\text{if $x \in \overline{X}$, $y \in \overline{Y}$} \\
       \text{$d_{Z}(\psi(x), \psi(y))$} &\quad\text{if $x, y \in \overline{Y}$}\\
     \end{cases}
     \]
     This defines a metric on $\overline{X} \sqcup \overline{Y}$ due to the fact that $\varphi(\overline{X})$ and $\psi(\overline{Y})$ are disjoint. This implies that 
     \[
     \tilde{d}_{GHB}(X, Y) \leq d_{GHB}(X, Y)+\epsilon.
     \]
     This completes the proof.
\end{proof}

Through Lemma~\ref{equivalence1}, triangle inequality with respect to $d_{GHB}$ is derived. Since the proof is a simple modification of the case of $d_{GH}$, we omit the proof. See the argument to prove triangle inequality in case of the Gromov-Hausdorff distance in \cite{BBI} and \cite{HKST}.
 
\begin{lemma}\label{triangle inequality}
	Let $X, Y, Z$ be noncomplete precompact metric spaces. Then
	\[
	d_{GHB}(X, Z) \leq d_{GHB}(X, Y)+ d_{GHB}(Y, Z).
	\]
\end{lemma}

Under the additional assumption that given spaces are locally compact, we get the following equivalence. 

\begin{lemma}\label{metric 2}
Let $X, Y$ be noncomplete precompact locally compact spaces. Then 
	$X$ is isometric to $Y$ if and only if $d_{GHB}(X, Y)=0$.
\end{lemma}
\begin{proof}
	We follow the case of the Gromov-Hausdorff distance, but with a modification. By Remark~\ref{isometry1}, if $X$ is isometric to $Y$, then $d_{GHB}(X, Y)=0$.  Now suppose $d_{GHB}(X, Y)=0$. By Lemma~\ref{equivalence1}, for each $k \in \mathcal{N}$, there exist metrics $d_{k}$ on $\overline{X} \sqcup \overline{Y}$ such that 
	\[
	d_{H}^{d_k}(\iota_{\overline{X}}(X), \iota_{\overline{Y}}(Y))+d_{H}^{d_k}(\iota_{\overline{X}}(\partial X), \iota_{\overline{Y}}(\partial Y))<1/k
	\]
	where $d_{H}^{d_k}$ is the Hausdorff distance on the metric space $(\overline{X} \sqcup \overline{Y}, d_{k})$. Then there exist maps(possibly not continuous) $I_k : \overline{X} \to \overline{Y}$ and $J_k : \overline{Y} \to \overline{X}$ such that for all $x \in \overline{X}$ and $y \in \overline{Y}$,
	\begin{equation}\label{close}
	d_{k}(\iota_{\overline{X}}(x), \iota_{\overline{Y}}(I_k(x)))<1/k \ \ \text{and} \ \ d_{k}(\iota_{\overline{Y}}(y), \iota_{\overline{X}}(J_k(y)))<1/k.
	\end{equation}
	Since $\overline{X}$ and $\overline{Y}$ are separable, we can take  $S_{X}:=(x_n)_n \subseteq  \overline{X}$ and  $S_{Y}:=(y_n)_n \subseteq \overline{Y}$ that are dense in $\overline{X}$ and $\overline{Y}$ respectively. For each $x \in S_{X}$, $(I_{k}(x))_k \in \overline{Y}$. The compactness of $\overline{Y}$ allows us to pick a subsequence $(I_{k_l}(x))_l$ such that $I_{k_l}(x) \to y_{x} \in \overline{Y}$ in $d_{Y}$. Doing the same argument for other elements in $S_{X}$ for the subsequence $(I_{k_l}(x))_l$ and by the diagonal argument, we can choose $(I_{k_l})_{l}$ such that $I_{k_l}(x) \to y_{x}$ in $d_{Y}$ for each $x \in S_{X}$. Define $I : S_{X} \to \overline{Y}$ by $S_{X} \ni x \mapsto y_{x} \in \overline{Y}$.  Since $d_{k_l}$ is admissible, for every $x, y \in S_{X}$ we have
	\[
	\begin{split}
		d_{Y}(I(x), I(y)) &=\lim\limits_{l \to \infty}d_{Y}(I_{k_l}(x), I_{k_l}(y))\\
		&=\lim\limits_{l \to \infty}d_{k_l}(\iota_{\overline{Y}}(I_{k_l}(x)), \iota_{\overline{Y}}(I_{k_l}(y)))\\
		&=d_X(x, y),
	\end{split}
	\]
	where the last equality holds due to \eqref{close}. Hence we have the canonical extension of $I$, denoted by $\tilde{I} : \overline{X} \to \overline{Y}$. Note that $\tilde{I}$ is a distance preserving map. By exactly the same argument for $J_{k} : \overline{Y} \to \overline{X}$, we get $\tilde{J} : \overline{Y} \to \overline{X}$. Since $\tilde{J} \circ \tilde{I}$ and $\tilde{I} \circ \tilde{J}$ are distance preserving self-maps of $\overline{X}$ and $\overline{Y}$ respectively, both $\tilde{I}$ and $\tilde{J}$ are bijective isometries. We will prove that $\tilde{I}(x) \in Y$ for every $x \in X$. To prove this, suppose that there exists $x \in X$ such that $\tilde{I}(x) \in \partial Y$.  Set $\alpha:=$dist$(x,  \partial X)>0$. Note that $\alpha$ is positive due to the local compactness of the space $X$. Let $\epsilon:=\alpha/4 >0$ be given. Then there exists $x_{\epsilon} \in S_{X}$ such that 
	\[
	d_{X}(x_{\epsilon}, x)=d_{Y}(\tilde{I}(x_{\epsilon}), \tilde{I}(x)) < \epsilon.
	\] 
	Hence dist$(x_{\epsilon}, \partial X) \geq 3\alpha/4$ and dist$(\tilde{I}(x_{\epsilon}), \partial Y) \leq \alpha/4$. For this $x_{\epsilon}$, there exists $N \in \mathbb{N}$ such that for each $k_l \geq N$, we have 
	\[
	d_{Y}(I_{k_l}(x_{\epsilon}), \tilde{I}(x_{\epsilon}))=d_{k_l}(\iota_{\overline{Y}}(I_{k_l}(x_{\epsilon})), \iota_{\overline{Y}}(\tilde{I}(x_{\epsilon}))) < \epsilon. 
	\]
	Also, by the construction of $I_{n_k}$, we know that 
	\[
	d_{k_l}(\iota_{\overline{X}}(x_{\epsilon}), \iota_{\overline{Y}}(I_{k_l}(x_{\epsilon})))<\frac{1}{k_l}.
	\]
	Combining these inequalities, for $k_l\geq N$, we have 
	\begin{align}\label{contradiction}
		 \text{dist}_{Y}(\tilde{I}(x_{\epsilon}), \partial Y) &\geq \text{dist}_{Y}(I_{k_l}(x_{\epsilon}), \partial Y)-\epsilon \notag \\
		&=\text{dist}_{d_{k_l}}(\iota_{\overline{Y}}(I_{k_l}(x_{\epsilon})), \iota_{\overline{Y}}(\partial Y))-\epsilon \notag \\
		&\geq \text{dist}_{d_{k_l}}(\iota_{\overline{Y}}(I_{k_l}(x_{\epsilon})), \iota_{\overline{X}}(\partial X))-\frac{1}{k_l}-\epsilon \notag \\
		&\geq \text{dist}_{d_{k_l}}(\iota_{\overline{X}}(x_{\epsilon}), \iota_{\overline{X}}(\partial X))-\frac{2}{k_l}-\epsilon \notag \\
		&\geq \text{dist}_{X}(x_{\epsilon}, \partial X)-\frac{2}{k_l}-\epsilon  \notag \\
		&\geq \alpha/2-\frac{2}{k_l},
	\end{align}
	where $\text{dist}_{d_{n_k}}$ represents the distance between a point and a set with respect to the metric $d_{n_k}$. Since $\text{dist}_{Y}(\tilde{I}(x_{\epsilon}), \partial Y)\leq \alpha/4$, by letting $l \to \infty$ in the inequality \eqref{contradiction}, we get a contradiction. Hence $\tilde{I}|_{X} : X \to Y$ is well-defined. By the similar argument we did above, we can also prove that $\tilde{I}(x) \in \partial Y$ for every $x \in \partial X$. This implies that $\tilde{I}|_{X}: X \to Y$ is bijective. This completes the proof. 
	%Proof of surjectivity of $I|_{X} : X \to Y$ goes like this. We only need to show that $\tilde{I}(x) \in \partial Y$ for every $x \in \partial X$. If there exists $x \in \partial X$ such that $\tilde{I}(x) \in Y$, then set $\alpha := d_{Y}(\tilde{I}(x), \partial Y)$ and take $x_{\epsilon} \in S_{X}$ such that $d_{X}(x, x_\epsilon) < \epsilon$, which also implies $d_{Y}(\tilde{I}(x), \tilde{I}(x_{\epsilon})) < \epsilon$. Hence we have $d(x_{\epsilon}, \partial X) \leq \epsilon=\alpha/4$ while $d(\tilde{I}(x_{\epsilon}), \partial Y) \geq 3\alpha/4$. Also there exists $N \in \mathbb{N}$ such that for any $k_l \geq N$, then $d_{Y}(I_{k_l}(x_{\epsilon}), \tilde{I}(x_{\epsilon}))=d_{k_l}(\iota_{\overline{Y}}(I_{k_l}(x_{\epsilon})), \iota_{\overline{Y}}(\tilde{I}(x_{\epsilon}))) < \epsilon$. Note also that $d_{d_{k_l}}(x_\epsilon, I_{k_l}(x_{\epsilon})) \leq 1/k_l$. Then $\alpha / 4 \geq d(x_{\epsilon}, \partial X) \geq \alpha/2 - 2/k_l$ by the same argument as in the proof. This is a contradiction.

	%$\tilde{J}|_{Y} : Y \to X$ is also well-defined. Moreover, since both  $(\tilde{J}|_{Y})\circ (\tilde{I}|_{X})$ and $(\tilde{I}|_{X})\circ (\tilde{J}|_{Y})$ are distance preserving self-maps to $X$ and $Y$ respectively, $\tilde{I}|_{X} : X \to Y$ is bijective. This completes the proof.
\end{proof}
\begin{rmk}\label{d_{GH} versus d_{GHB}}
	As we remarked in Section 1, it is immediate that for noncomplete precompact metric spaces $X$ and $Y$, we have 
	\[
	d_{GH}(\overline{X}, \overline{Y})+d_{GH}(\partial X, \partial Y) \leq d_{GHB}(X, Y).
	\]
	On the other hand, the following example shows that the the converse inequality does not hold. Set
	\[
	X:= [0, 1]\times[0, 1]\setminus \{(1/2, 0), \ (1/2, 1) \}
	\]
	and 
	\[
	Y:=[0, 1]\times[0, 1]\setminus \{(1, 0), \ (1, 1) \}.
	\]
	where the standard Euclidean metric is defined on both $X$ and $Y$. Then $d_{GH}(\overline{X}, \overline{Y})=d_{GH}(\partial X, \partial Y)=0$ but it follows that $d_{GHB}(X, Y)>0$ since if $d_{GHB}(X, Y)=0$, then by Lemma~\ref{metric 2} there exists an isometry between $X$ and $Y$. Under this isometry, the line 
	\[
	\{(1/2, t) \ :\ 0\leq t \leq 1 \} \ \subseteq \overline{X}
	\] 
	should get mapped to 
	\[
	\{(1, t) \ :\ 0\leq t \leq 1 \} \ \subseteq \overline{Y},
	\]
	 which is impossible by looking at the connected components. 
\end{rmk}

Set a class of metic spaces $\mathcal{M}$ by
\begin{equation}
 	\mathcal{M} := \left\lbrace X \;\middle|\;
  \begin{tabular}{@{}l@{}}
    $X$ is a noncomplete, precompact, and locally compact metric space. 
   \end{tabular}
 \right\rbrace / \sim. \notag
 \end{equation}
 where the equivalence relation $\sim$ between $X$ and $Y$ is defined by existence of an isometry between $X$ and $Y$. Then combining Lemma~\ref{triangle inequality} and Lemma~\ref{metric 2}, we get the following.
\begin{thm}[First assertion of Theorem~\ref{third theorem}]
	The space $\mathcal{M}$ is a metric space with the metric $d_{GHB}$.
\end{thm}
\begin{rmk}
	The assumption that metric spaces are locally compact in Theorem~\ref{third theorem} is essential in order to make $d_{GHB}$ a metric. One can see this by considering 
\[
X:=[0, 1] \cap \mathbb{Q} \ \ \text{and} \ \ Y:=[0, 1] \cap (\mathbb{R} \setminus \mathbb{Q}).
\]
Note that $d_{GHB}(X, Y)=0$ but there is no way to construct a bijective map between them.
\end{rmk}
\begin{rmk}
	 Note that $\mathcal{M}$ is not a complete metric space. Consider a sequence of metric spaces 
	\[
	X_{n}:=(0, 1)\times (0, 1)\setminus \bigcup_{k=1}^{n-1} [1/2, 1)\times \{k/n\}. 
	\]
	It is easy to verify that $X_{n}$ with the Euclidean metric are noncomplete precompact locally compact metric spaces. Moreover, $(X_{n})_n$ is Cauchy in $d_{GHB}$. If $\mathcal{M}$ is complete, we can find a limit $X$ such that $\overline{X}:=[0, 1]\times [0, 1]$, and  $d_{GHB}(X_n, X) \to 0$ as $n \to \infty$. However, $\partial X_n$ converges to  
	\[
	\partial X=[0, 1]\times [0, 1]\setminus (0, 1/2) \times (0, 1),
	\]
	which is not the boundary of the metric space $X$.
\end{rmk}

Here, in order to take the boundary into account, we define the notion of  correspondence with boundary.
\begin{defn}(Correspondence)
	Let $X, Y$ be metric spaces. A \emph{correspondence between $X$ and $Y$ with boundary} is a subset $R \subseteq \overline{X}\times\overline{Y}$ such that $\pi_{X}(R \cap (X\times Y))=X$, $\pi_{Y}(R \cap (X\times Y))=Y$, $\pi_{X}(R \cap (\partial X\times \partial Y))=\partial X$, and $\pi_{Y}(R \cap (\partial X\times \partial Y))=\partial Y$ where $\pi_{X} : \overline{X}\times \overline{Y} \to \overline{X}$ and $\pi_{Y} : \overline{X}\times \overline{Y} \to \overline{Y}$ are the canonical projections.
\end{defn}
\begin{defn}(Distortion)
	Let $X, Y$ be metric spaces and $R$ be a correspondence between $X$ and $Y$ with boundary. A \emph{distortion of $R$} is defined by
	\[
	\dis(R):=\sup\limits_{(x, y), (x', y') \in R}|d_{X}(x, x')-d_{Y}(y, y')|.%+\sup\limits_{(x, y), (x', y') \in B}|d_{X}(x, x')-d_{Y}(y, y')|.
	\]
\end{defn}
\begin{rmk}\label{dis to iso}
	One can check that if $\dis(R)=0$ for a given correspondence with boundary $R$, then it is not hard to show that there exists an isometry $f : \overline{X} \to \overline{Y}$ such that $f|_{X} : X \to Y$ and $f|_{\partial X} : \partial X \to \partial Y$.
\end{rmk}

\begin{prop}\label{dist-GHB}
	Let $X, Y$ be noncomplete precompact locally compact metric spaces. Then 
	\[
	\frac{1}{2}d_{GHB}(X, Y)\leq \frac{1}{2}\inf_{R}\dis (R)\leq d_{GHB}(X, Y),
	\]
	where the infimum is taken over all correspondences with boundary $R \subseteq X \times Y$.
\end{prop}
\begin{proof}
	 We will first prove that 
	\[
	d_{GHB}(X, Y) \geq \frac{1}{2}\inf_{R}\text{dis}(R).
	\]
	For each $r > d_{GHB}(X, Y)$, there exist a metric space $Z$ and isometric embeddings $\varphi : \overline{X} \to Z$ and $\psi : \overline{Y} \to Z$ such that 
	\[
	d^{Z}_{H}(\varphi(X), \psi(Y))+ d^{Z}_{H}(\varphi(\partial X), \psi(\partial Y))< r,
	\]
	where $d^{Z}_{H}$ is the Hausdorff distance in $Z$. Set $s>d^{Z}_{H}(\varphi(X), \psi(Y))$ and $t>d^{Z}_{H}(\varphi(\partial X), \psi(\partial Y))$ such that $s+t,r$. Define a correspondence with boundary
	\begin{equation}
 	R := \left\lbrace (x, y) \in \overline{X} \times \overline{Y} \;\middle|\;
  \begin{tabular}{@{}l@{}}
    $d_{Z}(\varphi(x), \psi(y)) \leq s$, $x \in X$, and $y \in Y$\\
    or \\
    $d_{Z}(\varphi(x), \psi(y)) \leq t$, $x \in \partial X$, and $y \in \partial Y$.
   \end{tabular}
 \right\rbrace. \notag
 \end{equation}
 If $(x, y), (x', y') \in R \cap (X\times Y)$, then 
 \[
 \begin{split}
 	|d_{X}(x, x')-d_{Y}(y, y')| &\leq |d_{Z}(\varphi(x), \varphi(x'))-d_{Z}(\varphi(x), \psi(y'))| \\
 	&\ \ \ + |d_{Z}(\varphi(x), \psi(y'))- d_{Z}(\psi(y), \psi(y'))|\\
 	&\leq d_{Z}(\varphi(x'), \psi(y'))+d_{Z}(\varphi(x), \psi(y))\\
 	&\leq 2s,
 \end{split}
 \]
 %which implies that $\sup\limits_{(x, y), (x', y') \in A}|d_{X}(x, x')-d_{Y}(y, y')|\leq 2s$.
  Similarly, if $(x, y), (x', y') \in R \cap (\partial X\times \partial Y)$, then
  \[
 |d_{X}(x, x')-d_{Y}(y, y')|\leq 2t,
 \]
 %which tells that $\sup\limits_{(x, y), (x', y') \in B}|d_{X}(x, x')-d_{Y}(y, y')|\leq 2t$.
 Also, if $(x, y) \in R \cap (X\times Y)$ and $(x', y') \in R \cap (\partial X\times \partial Y)$, then 
 \[
 |d_{X}(x, x')-d_{Y}(y, y')| \leq s+t.
 \]
Therefore, we conclude that
\[
\frac{1}{2}\inf\limits_{R}\text{dis}(R)\leq s+t < r.
\]
Letting $r \to d_{GHB}(X, Y)$, the claim is proved. Next we will prove
\[
 d_{GHB}(X, Y) \leq  \inf_{R}\text{dis}(R).
\]
Let $R$ be a correspondence with boundary and $r:=\frac{1}{2}\text{dis}(R)$.
 If $r=0$, by Remark~\ref{dis to iso}, we have an isometry $f : X \to Y$, which implies $d_{GHB}(X, Y)=0$ by Lemma~\ref{equivalence1}. Hence we assume $r>0$. Define an admissible metric $d$ on $\overline{X} \sqcup \overline{Y}$ by
\[
	d(\iota_{\overline{X}}(x), \iota_{\overline{Y}}(y))  :=
     \begin{cases}
       d_{X}(x, y) &\quad\text{if $x, y \in \overline{X}$}\\
       \text{$\inf\limits_{(x', y') \in R}(d_{X}(x, x')+d_{Y}(y, y')) + r$} &\quad\text{if $x \in \overline{X}$, $y \in \overline{Y}$} \\
        \text{$\inf\limits_{(x', y') \in R}(d_{X}(y, x')+d_{Y}(x, y')) + r$} &\quad\text{if $x \in \overline{Y}$, $y \in \overline{X}$}\\
       \text{$d_{Y}(x, y)$} &\quad\text{if $x, y \in  \overline{Y}$.}\\
     \end{cases}
     \]
One can check that $d$ satisfies triangle inequality due to the choice of  $r$. We claim that 
\[
d^{\overline{X}\sqcup \overline{Y}}_{H}(\iota_{\overline{X}}(X), \iota_{\overline{Y}}(Y))\leq r.
\]
For each $y \in Y$, there exists $x \in X$ such that $(x, y) \in R$ since $\pi_{Y}(R \cap (X \times Y))=Y$. Then 
\[
d(\iota_{\overline{X}}(x), \iota_{\overline{Y}}(y)) =\inf_{(x', y') \in R}(d_{X}(x, x')+d_{Y}(y, y'))+r\leq d_{X}(x,x)+d_{Y}(y,y)+r=r.
\]
By the same argument, for every $x \in X$, we can find $y \in Y$ such that 
$d(x, y) \leq r$,
which implies that $d^{\overline{X}\sqcup \overline{Y}}_{H}(X, Y) \leq r$.
Similarly, we have 
\[
d^{\overline{X}\sqcup \overline{Y}}_{H}(\iota_{\overline{X}}(\partial X), \iota_{\overline{Y}}(\partial Y))\leq r.
\] 
Hence we get
\[
\frac{1}{2}d_{GHB}(X, Y) \leq \frac{1}{2}(d^{\overline{X}\sqcup \overline{Y}}_{H}(\iota_{\overline{X}}(X), \iota_{\overline{Y}}(Y))+d^{\overline{X}\sqcup \overline{Y}}_{H}(\iota_{\overline{X}}(\partial X), \iota_{\overline{Y}}(\partial Y)))\leq r.
\]
We conclude that
\[
\frac{1}{2}d_{GHB}(X, Y) \leq \frac{1}{2}\inf\limits_{R}\text{dis}(R).
\]
\end{proof}

We introduce a notion of $\epsilon$-isometry with boundary.
\begin{defn}($\epsilon$-isometry with boundary)
	A map $f : \overline{X} \to \overline{Y}$ is called an $\epsilon$-isometry with boundary if
		\[
	\begin{split}
		\dis(f) &:=\sup\limits_{x, x' \in \overline{X}}|d_{Y}(f(x), f(x'))-d_{X}(x, x')| \leq \epsilon,
	\end{split}
	\]
	and
	\[
	Y \subseteq (f(X))_{\epsilon}, \ \ \text{and} \ \ \partial Y \subseteq (f(\partial X))_{\epsilon}.
	\]
\end{defn}

Here is a simple generalization of the relationship between $d_{GH}$ and $\epsilon$-isometry. We omit the proof since there is no change in it in comparison to the case of the Gromov-Hasudorff distance. This proposition will be used to prove the convergence of uniformized spaces.
\begin{prop}\label{isometry with boundary}
Let $X, Y$ be noncomplete precompact locally compact metric spaces. 
	\begin{enumerate}
		\item If $d_{GHB}(X, Y)<\epsilon$, then there exists a $2\epsilon$-isometry with boundary $f : \overline{X} \to \overline{Y}$.
		\item If there exists an $\epsilon$-isometry with boundary $f : \overline{X} \to \overline{Y}$, then $d_{GHB}(X, Y) \leq 3\epsilon$.
	\end{enumerate}
\end{prop}

Here we define the Gromov-Hausdorff convergence with boundary. As we will see later in Proposition~\ref{conv-dist-equivalent}, the convergence of spaces with respect to $d_{GHB}$ is equivalent to the notion defined below.
\begin{defn}(Gromov-Hausdorff convergence with boundary)\label{GHC with boundary}
	 Let $(X_n, d_n)_n$, $(X, d) \subseteq \mathcal{M}$. We say that \emph{$(X_n, d_n)_n$ is Gromov-Hausdorff convergent to $(X, d)$ with boundary} if there exist a metric space $Z$ and isometric embeddings $\varphi_n : \overline{X}_n \to Z$ and $\varphi : \overline{X} \to Z$ such that 
	 \[
	 d_{H}^{Z}(\varphi_n(X_n), \varphi(X)) \to 0
	 \]
	 and 
	 \[
	 d_{H}^{Z}(\varphi_n(\partial X_n), \varphi(\partial X)) \to 0
	 \]
	 where $d_{H}^{Z}$ is the Hausdorff distance on $Z$. 
	 \end{defn}

The following equivalence seems to be well-known in case of the Gromov-Hausdorff distance. Here we provide the proof for the reader's convenience.
\begin{prop}\label{conv-dist-equivalent}
	Let $(X_n, d_n)_n$, $(X, d) \subseteq \mathcal{M}$. Then the following notions of convergence are equivalent : 
	\begin{enumerate}
		\item $(X_n, d_n)_n$ is Gromov-Hausdorff convergent to $(X, d)$ with boundary in Definition~\ref{GHC with boundary}
		\item $d_{GHB}(X_n, X) \to 0$ \ as \ $n \to \infty$.
	\end{enumerate}
\end{prop}
\begin{proof}
	 Although the proof is the same one as the case of $d_{GH}$, we would like to record the proof for the sake of completeness. It is enough to prove that $d_{GHB}(X_n, X) \to 0$ as $n \to \infty$ implies $(X_n, d_n)_n$ is Gromov-Hausdorff convergent to $(X, d)$ with boundary in a sense of Definition~\ref{GHC with boundary}. We would like to remark that in this proof, we identify $\overline{X}_n$ and $ \overline{X}$ with the copies in $\overline{X}_n \sqcup \overline{X}$. Also, if $x \in \overline{X}$, then the corresponding element $\iota_{\overline{X}}(x)$ in $\overline{X}_n \sqcup \overline{X}$ by the canonical map $\iota_{\overline{X}} : \overline{X} \hookrightarrow \overline{X}_n \sqcup \overline{X}$ is still denoted by $x$. Same principle is applied to elements in $\overline{X}_n$.

	Let $\epsilon_n := d_{GHB}(X_n, X)+1/n$. Then By Lemma~\ref{equivalence1}, there exist admissible metrics $\delta_n$ on $\overline{X}_{n} \sqcup \overline{X}$ such that 
	\[
	d^{\delta_n}_{H}(X_n, X)+d^{\delta_n}_{H}(\partial X_n, \partial X)\leq \epsilon_n
	\]
	where $d^{\delta_n}_{H}$ is the Hausdorff distance on $(X_n \sqcup X, \delta_n)$. 
	Set $Z:=\overline{X} \sqcup \overline{X}_1 \sqcup \cdots \sqcup \overline{X}_n \sqcup \cdots := \overline{X}\times \{0\}\cup (\bigcup_{k=1}^{\infty} \overline{X}_k \times \{k\})$. We will define a metric $d_{Z}$ on $Z$ by
	 \[
	d_{Z}(x, y)  :=
     \begin{cases}
       \delta_{n}(x, y) &\quad\text{if $x, y \in \overline{X} \cup \overline{X}_n$}\\
       \text{$\inf_{w \in \overline{X}}(\delta_{n}(x, w)+\delta_{m}(w, y))$} &\quad\text{if $x \in \overline{X}_n$, $y \in \overline{X}_m$ with $n \neq m$}. \\
     \end{cases}
     \]
     Then we claim that this is a metric. Suppose $\delta(x, y)=0$. We only consider the case where $x \in \overline{X}_n$ and $y \in \overline{X}_{m}$. Take $(w_k)_k \subseteq \overline{X}$ such that 
     \[
     \delta_{n}(x, w_k)+\delta_{m}(w_k, y) \to 0,
     \]
      as $k \to \infty$. This implies that $x \in \overline{X}$ and $y \in \overline{X}$, which is not possible. Since the symmetry of this metric is obvious, it suffices to show that triangle inequality holds. We check cases $x \in \overline{X}_n$, $y \in \overline{X}_m$, $z \in \overline{X}$ with $n \neq m$ and $x \in \overline{X}_n$, $y \in \overline{X}_m$, $z \in \overline{X}_l$ with $n, m, l$ not all distinct since all other cases are straightforward. \\
      (\textbf{Case 1}) $x \in \overline{X}_n$, $y \in \overline{X}_m$, and $z \in \overline{X}$. Since $\overline{X}$ is compact, we can pick a point $w \in \overline{X}$ such that 
      \[
      d_{Z}(x, y)=\delta_{n}(x, w)+\delta_{m}(w, y).
      \]
      Then 
      \[
      \begin{split}
      	d_{Z}(x, y)+d_{Z}(y, z) &= \delta_{n}(x, w)+\delta_{m}(w, y)+\delta_{m}(y, z)\\
      	&\geq \delta_{n}(x, w)+\delta_{m}(w, z)=d_{Z}(x, z)
      \end{split}
      \] 
      where we used the fact that $\delta_m(w, z)=\delta_n(w, z)$ for every $w, z \in \overline{X}$.\\
      (\textbf{Case 2}) $x \in \overline{X}_n$, $y \in \overline{X}_m$, and $z \in \overline{X}_l$. Take $w \in \overline{X}$ and $\bar{w} \in \overline{X}$ such that
     \[
     d_{Z}(x, y)=\delta_n(x, w)+\delta_m(w, y)
     \]
     and
     \[
     d_{Z}(y, z)=\delta_m(y, \bar{w})+\delta_l(\bar{w}, z).
     \]
     Then,
     \[
     \begin{split}
     	d_{Z}(x, z) &\leq \delta_{n}(x, \bar{w})+\delta_l(\bar{w}, z)\\
     	&\leq \delta_{n}(x, w)+\delta_{n}(w, \bar{w})+\delta_l(\bar{w}, z)\\
     	&= \delta_{n}(x, w)+\delta_{m}(w, \bar{w})+\delta_l(\bar{w}, z)\\
     	&\leq \delta_{n}(x, w)+\delta_{m}(w, y)+\delta_m(y, \bar{w})+\delta_l(\bar{w}, z)\\
     	&=d_{Z}(x, y)+d_{Z}(y, z).
     	\end{split}
     \]
     Hence $d_{Z}$ is a metric on $Z$. By taking a metric space $(Z, d_{Z})$, $\varphi_n:=\iota_{\overline{X}_n} :  \overline{X}_n \hookrightarrow Z$, and $\varphi:=\iota_{\overline{X}} :  \overline{X} \hookrightarrow Z$, we have 
     \[
     d^{Z}_{H}(\varphi_n(X_n), \varphi(X))+d^{Z}_{H}(\varphi_n(\partial X_n), \varphi(\partial X))\leq \epsilon_n \to 0
     \]
     as $n \to \infty$. This completes the proof.
\end{proof}

\section{An application to a class of bounded $A$-uniform spaces}
In this section we will present an application of the Gromov-Hausdorff distance with boundary $d_{GHB}$ to a class of bounded $A$-uniform spaces and prove Theorem \ref{third theorem}. Here we will first prove the stability of $A$-uniform spaces with respect to $d_{GHB}$. Recall from Definition~\ref{uniform space} that a noncomplete locally compact metric space $(\Omega, d)$ is called an $A$-uniform space if for each pair of points $x, y \in \Omega$ there exists a curve $\gamma$ from $x$ to $y$ such that
\begin{enumerate}
		\item $l_{d}(\gamma) \leq A d(x, y)$,
		\item $l_{d}(\gamma|_{[0, t]})\wedge l_{d}(\gamma|_{[t, 1]}) \leq A$ dist$(\gamma(t), \partial \Omega)$ \ \ \ for every $t \in [0, 1]$.
	\end{enumerate}
	%\begin{rmk}\label{the choice of parametrization}
	%	Notice that we do not take the arc-length parametrization for an $A$-uniform curve $\gamma'$ due to the above definition. In this section, we always take the parametrization set by the following procedure. First consider the arc-length parametrization $\gamma' : [0, l_{d}(\gamma')] \to \Omega$ and define $\alpha(t):=l_{d}(\gamma') t$ for every $t \in [0, 1]$. Then we set $\gamma := \gamma' \circ \alpha : [0, 1] \to \Omega$, which is the curve considered in this section. Note also that the choice of $\alpha$ depends on the curve $\gamma'$.
	%\end{rmk}

Before stating the next lemma, we would like to remark that by \cite[Proposition 2.20]{BHK}, every bounded $A$-uniform space is precompact. Hence the isometry class of a bounded $A$-uniform space is in $\mathcal{M}$.
\begin{lemma}\label{uniform space stable}
	Let $(\Omega_n, d_n)_n \subseteq \mathcal{M}$ be a sequence of bounded $A$-uniform spaces and $(\Omega, d) \in \mathcal{M}$. Suppose $(\Omega_n, d_n)_n$ Gromov-Hausdorff converges to $(\Omega, d)$ with boundary. Then $(\Omega, d)$ is also an $A$-uniform space.
\end{lemma}
\begin{proof}
	 First note that in this proof we do not take the arc-length parametrization for $A$-uniform curves $\gamma'$. We always take the parametrization set by the following procedure. First consider the arc-length parametrization $\gamma' : [0, l_{d}(\gamma')] \to \Omega$ and define $\alpha(t):=l_{d}(\gamma') t$ for every $t \in [0, 1]$. Then we set $\gamma := \gamma' \circ \alpha : [0, 1] \to \Omega$, which is the curve considered in this proof. By Proposition~\ref{conv-dist-equivalent}, we may assume that $(\Omega_n)_n$ and $\Omega$ are all subsets of the same space $(Z, d)$. Recall again that $\Omega_n$ and $\Omega$ need not be open in $(Z, d)$. Set $\epsilon_n > d^{Z}_{H}(\Omega_n, \Omega)$. Let $x, y \in \Omega$  be given. Take $x_n, y_n \in \Omega_n$ such that 
	\[
	d(x, x_n)\leq \epsilon_n \ \ \ \text{and} \ \ \ d(y, y_n)\leq \epsilon_n.
	\]
	Take $A$-uniform curves $\gamma_n : [0, 1] \to \Omega_n$ from $x_n$ to $y_n$ parametrized as described above. We record the following inequality for the later use. For each $t, t' \in [0, 1]$, we get
	\begin{align}\label{au}
		d(\gamma_n(t), \gamma_n(t')) &\leq l_{d_n}(\gamma_n|{[t, t']})\notag  \\
		&= l_{d_n}(\gamma_n) |t-t'|\notag \\
		&\leq Ad(x_n, y_n)|t-t'|\notag \\
		&\leq A(d(x, y)+2\epsilon_n)|t-t'|.
	\end{align}
	Let $t_1 \in [0, 1]\cap \mathbb{Q}$ and $n \in \mathbb{N}$ be given. Then there exists $y_n(t_1) \in \Omega$ such that $d(y_n(t_1), \gamma_n(t_1))\leq \epsilon_n$. Since we know that $\overline{\Omega}$ is compact, we can take a subsequence $(y_{n_k}(t))_k$ convergent to some $y_t \in \overline{\Omega}$. By a standard diagnalization argument we may assume that for every $t \in [0, 1] \cap \mathbb{Q}$, there exists $y_t \in \overline{\Omega}$ such that $d(y_t, \gamma_{n_k}(t)) \to 0$ as $k \to \infty$. Note that, by triangle inequality and \eqref{au},
	\begin{align}
		d(y_t, y_{t'}) &\leq d(y_t, y_{n_k}(t))+d(y_{n_k}(t), \gamma_{n_k}(t))+d(\gamma_{n_k}(t), \gamma_{n_k}(t'))\notag \\
		&\ \ \ +d(\gamma_{n_k}(t'), y_{n_k}(t'))+ d(y_{n_k}(t'), y_{t'}) \notag \\
		&\leq 2 \epsilon_{n_k}+d(y_t, y_{n_k}(t))+d(\gamma_{n_k}(t), \gamma_{n_k}(t'))+d(y_{n_k}(t'), y_{t'}) \notag \\
		&\leq 2 \epsilon_{n_k}+d(y_t, y_{n_k}(t))+A(d(x, y)+2\epsilon_n)|t-t'|+d(y_{n_k}(t'), y_{t'}) \notag
		%&\leq 2 \epsilon_{n_k}+d(y_t, y_{n_k}(t))+d(\gamma'_{n_k}(\alpha(t)), \gamma'_{n_k}(\alpha(t')))+d(y_{n_k}(t'), y_{t'}) \notag \\
		%&\leq 2 \epsilon_{n_k}+d(y_t, y_{n_k}(t))+|\alpha(t)-\alpha(t')|+d(y_{n_k}(t'), y_{t'}) \notag \\
		%&\leq 2 \epsilon_{n_k}+d(y_t, y_{n_k}(t))+l_{n_k}(\gamma_{n_k})|t-t'|+d(y_{n_k}(t'), y_{t'}) \notag \\
		%&\leq 2 \epsilon_{n_k}+d(y_t, y_{n_k}(t))+A(d(x, y)+1)|t-t'|+d(y_{n_k}(t'), y_{t'})\notag 
	\end{align}
	for every $t, t' \in [0, 1]\cap \mathbb{Q}$. By taking the limit as $k \to \infty$ in the above inequality, we get 
	\[
	d(y_t, y_{t'}) \leq Ad(x, y)|t-t'|,
	\]
	which implies that the map $t \in [0, 1]\cap \mathbb{Q} \mapsto y_t \in \overline{\Omega}$ is $Ad(x, y)$-lipschitz. Hence, by canonical extension to the interval $[0, 1]$, we have an $Ad(x, y)$-lipschitz map $\gamma : [0, 1] \to \overline{\Omega}$ such that $\gamma(t)=y_t$ for every $t \in [0, 1]\cap \mathbb{Q}$. We also note that $\gamma_{n_k}(t) \to \gamma(t)$ as $k \to \infty$ for every $t \in [0, 1]$ since 
	\begin{align}
		\limsup\limits_{k \to \infty}d(\gamma_{n_k}(t), \gamma(t)) &\leq \limsup\limits_{k \to \infty}d(\gamma_{n_k}(t), \gamma_{n_k}(t_j))+\limsup\limits_{k \to \infty}d(\gamma_{n_k}(t_j), \gamma(t_j))+d(\gamma(t_j), \gamma(t))\notag \\
		&\leq 2Ad(x,y)|t-t_j|\notag 
	\end{align} 
	  for $t_j \in [0, 1] \cap \mathbb{Q}$. Letting $t_j \to t$ shows that $\gamma_{n_k}(t) \to \gamma(t)$. We claim that the curve $\gamma$ is $A$-uniform. Since $\gamma$ is $Ad(x, y)$-lipschitz, we know that $l_d(\gamma) \leq Ad(x, y)$. Next we will prove that for each $t \in [0, 1]$,
	\[
 	 l_{d}(\gamma|_{[0, t]})\wedge l_{d}(\gamma|_{[t, 1]}) \leq A \ \text{dist} (\gamma(t), \partial \Omega).
	\]
	Since $\gamma_{n_k}$ are all $A$-uniform curves, we know that 
	\begin{align}\label{au1}
		l_{d}(\gamma_{n_k})(t\wedge (1-t))=l_{d_{n_k}}(\gamma_{n_k}|_{[0, t]})\wedge l_{d_{n_k}}(\gamma_{n_k}|_{[t, 1]}) &\leq A \ \text{dist} (\gamma_{n_k}(t), \partial \Omega_{n_k}) \notag \\
		&\leq A \ \text{dist} (\gamma_{n_k}(t), \partial \Omega)+A\tilde{\epsilon}_{n_k} 
	\end{align}
	where $\tilde{\epsilon}_{n_k}:= d^{Z}_{H}(\partial \Omega_{n_k}, \partial \Omega)$.  Since length does not increase under uniform limits of curves, we get 
	\begin{equation}\label{au2}
		\limsup\limits_{k \to \infty}l_{d_{n_k}}(\gamma_{n_k}|_{[0, t]})\wedge l_{d_{n_k}}(\gamma_{n_k}|_{[t, 1]}) \leq A \ \text{dist} (\gamma(t), \partial \Omega).
	\end{equation}
	Since the curve $\gamma$ is a uniform limit of $(\gamma_{n_k})_k$, we can get
	\begin{equation}\label{au3}
	l_{d}(\gamma|_{[0, t]})\wedge l_{d}(\gamma|_{[t, 1]}) \leq \limsup\limits_{k \to \infty}l_{d_{n_k}}(\gamma_{n_k}|_{[0, t]})\wedge l_{d_{n_k}}(\gamma_{n_k}|_{[t, 1]})
	\end{equation}
	for each $t \in [0, 1]$. By inequalities \eqref{au2} and \eqref{au3}, we conclude that
	\[
	l_{d}(\gamma|_{[0, t]})\wedge l_{d}(\gamma|_{[t, 1]}) \leq  A \ \text{dist} (\gamma(t), \partial \Omega).
	\]
	This tells us that even though initially the curve $\gamma$ was defined in $\overline{\Omega}$, indeed $\gamma$ is in $\Omega$ from the last inequality. This completes the proof.
\end{proof}
%\begin{remark}
	%The first part of the above proof implies that if a sequence of compact $A$-quasiconvex metric spaces $X_n$ converges to some $X$ in the Gromov-Hausdorff topology, then the limit space $X$ is also quasiconvex
%\end{remark}

\begin{defn}
	Let $(X, d)$ be a metric space and set
	\begin{equation}
 	\mathcal{U}_{X}(A, R) := \left\lbrace \Omega \subseteq X \;\middle|\;
  \begin{tabular}{@{}l@{}}
    $\Omega$ is a bounded $A$-uniform space and $\diam (\Omega) \geq R$
   \end{tabular}
 \right\rbrace. \notag
 \end{equation}
\end{defn}
\begin{rmk}
	Since our focus is uniform spaces, we do not require $\Omega \subseteq X$ to be open in $X$.
\end{rmk}

\begin{prop}\label{HB-basics}
Let $(X, d)$ be a compact metric space and $A \geq 1$ and $R>0$ be fixed. Then $(\mathcal{U}_{X}(A, R), d_{HB})$ is a compact metric space.
\end{prop}
\begin{proof}
	First we prove that $\mathcal{U}_{X}(A, R)$ is complete with respect to $d_{HB}$. Let $(\Omega_n)_n$ be a Cauchy sequence in $(\mathcal{U}_{X}(A, R), d_{HB})$. Note that $\overline{\Omega}_n$ and $\partial \Omega_n$ are all compact.  Since $(\overline{\Omega}_n)_n$ and $(\partial \Omega_n)_n$ are Cauchy sequences with respect to the Hausdorff distance $d^{X}_{H}$, there exist compact sets $S, T \subseteq X$ such that $d^{X}_{H}(\overline{\Omega}_n, S) \to 0$, $d^{X}_{H}(\partial \Omega_n, T) \to 0$ as $n \to \infty$ by \cite[Theorem 7.3.8]{BBI}. Note that $T \subseteq S$. Set $\Omega := S \setminus T$. We now prove that $\Omega$ is a nonempty set. Suppose $\Omega = \emptyset$. Then we have $S=T$. For each $n \in \mathbb{N}$, we can take a pair of points $x_n, y_n \in \Omega_n$ with $d(x_n, y_n) \geq R/2$ and an arc-length parametrized $A$-uniform curve $\gamma_n : [0, l_{d}(\gamma_n)] \to \Omega_n$ from $x_n$ to $y_n$. Then, by $A$-uniformity, we get
	\[
	\frac{R}{4} \leq A \text{dist}(\gamma_n(R/4), \partial \Omega_n) \to 0
	\]
	as $n \to \infty$, which is a contradiction. Hence $\Omega$ is nonempty.
	We next claim that $\Omega$ is a dense subset of $S$. Suppose that there exists $x \in T$ such that $B_d(x, r)\cap S \subseteq T$ for some $r>0$. We may assume that $r < R/2$.
	Set $\epsilon := \frac{r}{16A}$. 
	Then there exists $N_{\epsilon}\in \mathbb{N}$ such that for $n \geq N_{\epsilon}$, we get $d^{X}_{H}(\Omega_n, S)+d^{X}_{H}(\partial \Omega_n, T)< \epsilon$. We claim that for every point $z_n \in B_d(x, r/2)\cap \Omega_n$, dist$(z_n, T) \leq \epsilon$. In fact, we can take $z' \in S$ such that $d(z', z_n) \leq \epsilon$. Also,
	 \[d(x, z') \leq d(x, z_n)+d(z_n, z') \leq r/2+\epsilon < r.
	 \]
	  Hence $z' \in B_d(x, r)\cap S \subseteq T$, which implies that 
	  \[
	  \text{dist}(z_n, T) \leq d(z_n, z') \leq \epsilon.
	  \]  
	   Take $x_n \in \Omega_n$ satisfying $d(x_n, x) < \epsilon$. Since $\text{diam}(\Omega_n) \geq R$ and $r<R/2$, we have $\Omega_n \neq B_d(x_n, r) \cap \Omega_n$. Hence we can pick $y_n \in \Omega_n$ such that $d(x_n, y_n)\geq r$. Let $\gamma_n : [0, l_d(\gamma_n)] \to \Omega_n$ be an $A$-uniform curve in $\Omega_n$ from $x_n$ to $y_n$. Set $t : =r/4$. Noting that $\gamma_n(t) \in B_d(x, r/2)\cap \Omega_n$ and the above claim, we get
	\[
	t=l_d(\gamma_n|_{[0, t]})\wedge l_d(\gamma_n|_{[t, l_d(\gamma_n)]}) \leq A \text{dist}(\gamma_n(t), \partial \Omega_n) \leq A (\text{dist}(\gamma_n(t), T)+\epsilon) \leq 2A \epsilon \leq \frac{r}{8},
	\]
	which is a contradiction. Hence $\Omega$ is dense in $S$ and noncomplete. Local compactness of $\Omega$ follows directly from the fact that $d(x, T)>0$ for each $x \in \Omega$ and $S$ and $T$ are compact. From Lemma~\ref{uniform space stable}, we conclude that $\Omega$ is an $A$-uniform space. Therefore $\mathcal{U}_{X}(A, R)$ is complete. Regarding the compactness of $\mathcal{U}_{X}(A, R)$, since $X$ is compact, for every sequence $(\Omega_n)_n$ from $\mathcal{U}_{X}(A, R)$ we can take out a subsequence $(\Omega_{n_k})_k$ such that both $(\overline{\Omega}_{n_k})_k$ and $(\partial \Omega_{n_k})_k$ are Cauchy in $d_{H}$, which means $(\Omega_{n_k})_k$ is Cauchy in $d_{HB}$. Since $\mathcal{U}_{X}(A, R)$ is complete, there exists a limit $\Omega \in \mathcal{U}_{X}(A, R)$ such that $d_{HB}(\Omega_{n_k}, \Omega) \to 0$ as $k \to \infty$. 
\end{proof}

\begin{defn}
	Let $A \geq 1$ and $R>0$. Define 
	\begin{equation}
 	\mathcal{U}(A, R) := \left\lbrace [\Omega] \in \mathcal{M} \;\middle|\;
  \begin{tabular}{@{}l@{}}
    $\Omega$ is a bounded $A$-uniform space and $\text{diam} (\Omega) \geq R$
   \end{tabular}
 \right\rbrace \notag
 \end{equation}
 where $[\Omega]$ is an equivalence class of $\Omega$.
\end{defn}

\begin{thm}[Second assertion of Theorem~\ref{third theorem}]
	The metric space $(\mathcal{U}(A, R), d_{GHB})$ is a complete metric space.
\end{thm}
\begin{proof}
	Let $(\Omega_n)_n$ be a Cauchy sequence of uniform spaces with respect to $d_{GHB}$. Note that $(\overline{\Omega_n})_n$ is Cauchy with respect to $d_{GH}$ and hence it is totally bounded. By Proposition~\ref{Gromov embedding} there exists a compact set $K \subseteq l^{\infty}$ such that every $\overline{\Omega}_n$ admits an isometric embedding into $K$. Since $\mathcal{U}_{K}(A, R)$ is compact by Proposition~\ref{HB-basics},  there exists a subsequence $(\Omega_{n_k})_k$ and $\Omega \in \mathcal{U}_{K}(A, R)$ such that $\Omega_{n_k} \to \Omega$ in $d_{HB}$ as $k \to \infty$, which implies $d_{GHB}(\Omega_{n_k}, \Omega) \to 0$ as $k \to \infty$. Since the original sequence is Cauchy, $d_{GHB}(\Omega_{n}, \Omega) \to 0$ as $n \to \infty$.
\end{proof}

\begin{thm}\label{Compactness for uniform spaces}(Compactness theorem for uniform spaces)
	Let $(\Omega_n)_n \subseteq \mathcal{U}(A, R)$. If one of the conditions in Proposition~\ref{cap-cov} holds for $(\overline{\Omega}_n)_n$, then there exists a subsequence $(\Omega_{n_k})_k$ such that $\Omega_{n_k} \to \Omega \in \mathcal{U}(A, R)$ as $k \to \infty$ with respect to $d_{GHB}$.
\end{thm}
\begin{proof}
	Since $(\overline{\Omega}_n)_n$ is totally bounded with respect to $d_{GH}$, by Proposition~\ref{Gromov embedding} there exists a compact set $K \subseteq l^{\infty}$ such that every $\overline{\Omega}_n$ admits an isometric embedding into $K$. By Proposition~\ref{HB-basics}, there exists $(\Omega_{n_k})_k$ such that $\Omega_{n_k} \to \Omega \in \mathcal{U}_{K}(A, R)$ with respect to  $d_{HB}$, which implies that $\Omega_{n_k} \to \Omega$ with respect to $d_{GHB}$. 
\end{proof}
\begin{rmk}
	A type of compactness theorem for a sequence of noncomplete precompact metric spaces has also been proved in \cite[Theorem 4.1]{PS}. It was proven there that for a sequence of noncomplete precompact metric spaces $(X_n)_n$ converging to a limit space $X$ and a given $(\delta_i)_i$ monotonically decreasing to $0$, there exists a subsequence $(X_{n_k})_k$ such that $\delta_i$-inner regions converge to the set inside the limit space $X$ for all $i \in \mathbb{N}$. Note that a $\delta$-inner region is the set of all points $\delta$-away from  the boundary. Our result, however, states that under the additional assumption that metric spaces are $A$-uniform, we can find a limit space such that the metric spaces and their boundaries converge to the limit and its metric boundary, respectively. 
\end{rmk}
\begin{rmk}\label{PI space}
	There is a remark on stability of noncomplete PI spaces. Here we say that a metric measure space $(X, d, \mu)$ is a PI space if $\mu$ is a doubling measure and $X$ supports a Poincar\'e inequality. We refer readers to \cite{HK} for the exact definition of PI spaces. 
	We say that $(X_n, d_n, \mu_n)_n \subseteq \mathcal{M}$ measured Gromov-Hausdorff converges to $(X, d, \mu) \in \mathcal{M}$ with boundary if there exist a metric space $(Z, d_{Z})$ and isometric embeddings $\iota_n : \overline{X}_n \hookrightarrow Z$,  $\iota : \overline{X} \hookrightarrow Z$ such that
	\begin{itemize}
		\item $d_{H}^{Z}(\iota_n(X_n), \iota(X))+d_{H}^{Z}(\iota(\partial X_n), \iota(\partial X)) \to 0$ as $n \to \infty$,
		\item $\iota_{n, *}\mu_n \rightharpoonup \iota_{*}\mu$ (weak* convergence in $C_b(Z)^{*}$),
	\end{itemize}
	where $C_b(Z)$ is the set of all continuous functions with bounded support and $\iota_{n, *}\mu_n$, $\iota_{*}\mu$ are pushforward measures of $\mu_n$ and $\mu$, respectively. For a given sequence of bounded PI uniform spaces $(X_n, d_n, \mu_n)_n$ measured Gromov-Hausdorff converging to $(X, d, \mu) \in \mathcal{M}$ with boundary, we claim that $(X, d, \mu)$ is a PI space. In fact, we have already shown, by Theorem \ref{third theorem}, that $X$ is a uniform space. Additionally, the sequence $(\overline{X}_n, d_n, \mu_n)$ is a PI space by \cite[Lemma 8.2.3]{HKST}. Moreover, by \cite[Chapter 9]{Ch} and \cite[Theorem 3]{Kei}, $(\overline{X}, d, \mu)$ supports a Poincar\'e inequality. Since the limit space $X$ is a uniform domain in the PI space $\overline{X}$, we conclude that $X$ supports a Poincar\'e inequality by \cite[Theorem 4.4]{BSh}.
\end{rmk}

%\begin{rmk}
	%We can also have a type of compactness theorem for the class $\mathcal{M}$. First notice that we did not use the quasiconvexity of $A$-uniform curves to prove Proposition~\ref{HB-basics}. Hence if a sequence $(X_n, d_n)_n \in \mathcal{M}$ is such that
	%\begin{enumerate}
		%\item any of the conditions in Lemma~\ref{cap-cov} holds for $(\overline{X}_n, d_n)_n$,
		%\item there exists $R>0$ such that $\diam X_n \geq R$ for all $n \in \mathbb{N}$,
		%\item there exists $A>0$ such that for each $n \in \mathbb{N}$ and  every pair of points $x, y \in X_n$, there exists a curve $\gamma : [0, 1] \to X_n$ from $x$ to $y$ such that 
	%\[
	%l_{d}(\gamma|_{[0, t]})\wedge l_{d}(\gamma|_{[t, 1]}) \leq A  \dist(\gamma(t), \partial X_n) \ \ \ \text{for every} \ t \in [0, 1].\]
	%\end{enumerate} 
	%Then there exist a subsequence $ (X_{n_k}, d_{n_k})_k$ and a limit space $(X, d) \in \mathcal{M}$ such that $d_{GHB}(X_{n_k}, X) \to 0$ as $k \to 0$.
%\end{rmk}

\section{Stability of $M$-roughly starlike $\delta$-Gromov hyperbolic spaces}
In this section we prove the stability of a sequence of $M$-roughly starlike $\delta$-Gromov hyperbolic spaces, under the assumption that a limit exists. Recall again that curves are parametrized by arclength with respect to $d$. We first prove that $\delta$-Gromov hyperbolicity is stable under the pointed Gromov-Hausdorff convergence.
\begin{prop}\label{GH-limit}
Let  $(X_n, d_n, p_n)_n$ be a sequence of pointed $\delta$-Gromov hyperbolic  spaces. Suppose that $(X_n, d_n, p_n)_n$ is pointed Gromov-Hausdorff convergent to $(X, d, p)$ as $n \to \infty$. Then $X$ is a $\delta$-Gromov hyperbolic space.
\end{prop}
\begin{proof}
	Let $\epsilon>0$ be fixed. Pick arbitrary four points $x, y, z, m \in X$. Then let $R>0$ so that all of these four points are in $B_d(p, R-\epsilon)$. From Property (3) in Definition~\ref{p-GH}, we know that for sufficiently large $n \in \mathbb{N}$, there exist $m_{n}, x_{n}, y_{n}, z_{n} \in X_{n}$ such that
	\[
	d(f^{\epsilon}_{n}(m_{n}), m)<\epsilon , \ \ d(f^{\epsilon}_{n}(x_{n}), x)<\epsilon , \ \ d(f^{\epsilon}_{n}(y_{n}), y)<\epsilon , \ \ d(f^{\epsilon}_{n}(z_{n}), z)<\epsilon.
	\]
	 Since $X_{n}$ is a $\delta$-Gromov hyperbolic space, we have
	 \[
	 (x_{n}|z_{n})_{m_{n}}\geq (x_{n}|y_{n})_{m_{n}}\wedge(y_{n}|z_{n})_{m_{n}}-\delta.
	 \]
	  From Property (2) in Definition \ref{p-GH}, we notice that 
	  \[
	  (x_{n}|z_{n})_{m_{n}}\leq (f^{\epsilon}_{n}(x_{n})|f^{\epsilon}_{n}(z_{n}))_{f^{\epsilon}_{n}(m_{n})}+\frac{3}{2}\epsilon.
	  \]
	  Similarly, we can get
	  \[
	  (x_{n}|y_{n})_{m_{n}}\wedge(y_{n}|z_{n})_{m_{n}}\geq (f^{\epsilon}_{n}(x_{n})|f^{\epsilon}_{n}(y_{n}))_{f^{\epsilon}_{n}(m_{n})}\wedge(f^{\epsilon}_{n}(y_{n})|f^{\epsilon}_{n}(z_{n}))_{f^{\epsilon}_{n}(m_{n})}-\frac{3}{2}\epsilon
	  \]
	   Combining all of these inequalities, we have
	   \[
	   \begin{split}
	   	(f^{\epsilon}_{n}(x_{n})|f^{\epsilon}_{n}(z_{n}))_{f^{\epsilon}_{n}(m_{n})}+\frac{3}{2}\epsilon &\geq (f^{\epsilon}_{n}(x_{n})|f^{\epsilon}_{n}(y_{n}))_{f^{\epsilon}_{n}(m_{n})}\wedge(f^{\epsilon}_{n}(y_{n})|f^{\epsilon}_{n}(z_{n}))_{f^{\epsilon}_{n}(m_{n})}\\
	   	&\ \ \ \ -\frac{3}{2}\epsilon - \delta.
	   \end{split}
	   \]
	   Also, from direct calculation, we have
	   \[
	   (f^{\epsilon}_{n}(x_{n})|f^{\epsilon}_{n}(z_{n}))_{f^{\epsilon}_{n}(m_{n})}\leq (x|z)_{m}+3\epsilon/2,
	   \]
	   and
	   \[
	   (f^{\epsilon}_{n}(x_{n})|f^{\epsilon}_{n}(y_{n}))_{f^{\epsilon}_{n}(m_{n})}\wedge(f^{\epsilon}_{n}(y_{n})|f^{\epsilon}_{n}(z_{n}))_{f^{\epsilon}_{n}(m_{n})}\geq (x|y)_{m}\wedge(y|z)_{m}-3\epsilon/2.
	   \]
	    We conclude that
	    \[
	    (x|z)_{m}+3\epsilon \geq (x|y)_{m}\wedge(y|z)_{m}-3\epsilon-\delta.
	    \]
	     Letting $\epsilon$ tend to 0 completes the proof.
	 \end{proof}
	 
We next prove that the roughly starlike property is stable. 
\begin{prop}\label{roughly starlike limit}
	Let $(X_n, d_n, p_n)_n$ be a sequence of pointed proper geodesic $M$-roughly starlike $\delta$-Gromov hyperbolic spaces which is pointed Gromov-Hausdorff convergent to some metric space $(X, d, p)$. Then, there exists $\tilde{M}:=\tilde{M}(M, \delta)$ such that $(X, d, p)$ is $\tilde{M}$-roughly starlike.
\end{prop}
\begin{proof}
	For each $k \in \mathbb{N}$, set $\epsilon := 1/k$ and $r:= k$. Then there exists $n_{k} \in \mathbb{N}$ such that $f^{1/k}_{n_{k}} : B(p_{n_{k}}, k)\to X$ with the properties in Definition~\ref{p-GH}. For each $x \in X$, take $k_x \in \mathbb{N}$ such that $k_x \geq d(p, x)+1$. Note that we can find $x_{n_{k}}\in B_{d_{n_k}}(p_{n_{k}},k)$ such that $d(f^{1/k}_{n_{k}}(x_{n_{k}}), x)\leq 1/k$ for each $k \geq k_x$. Since $X_{n_{k}}$ is $M$-roughly starlike and $x_{n_{k}} \in X_{n_{k}}$, there is a geodesic ray $\gamma_{{n}_{k}} \subseteq X_{n_{k}}$ such that $\gamma_{{n}_{k}}(0)=p_{n_{k}}$ and $\text{dist}(x_{n_{k}}, \gamma_{{n}_{k}})\leq M$. Note that since $\gamma_{{n}_{k}}$ is a geodesic ray, $\gamma_{{n}_{k}}([0, k-1]) $ is in $ B_{d_{n_k}}(p_{n_{k}}, k)$ and $f^{1/k}_{n_{k}} \circ \gamma_{{n}_{k}}|_{[0, k-1]}$ is a $(1, 1/k)$-quasi-isometric path(see Definition \ref{}). Since the limit space $X$ is geodesic, we have a geodesic curve $\tilde{\gamma}_{n_{k}}$ emanating from $p$ to $f^{1/k}_{n_{k}}(\gamma_{{n}_{k}}(k-1))$. Note that the limit space is proper geodesic $\delta$-Gromov hyperbolic by Proposition~\ref{GH-limit}. Therefore from the geodesic stability theorem \cite[Theorem 3.7]{Jussi} we get 
	\begin{equation}\label{stab}
		d_{H}(\tilde{\gamma}_{n_{k}}, f^{1/k}_{n_{k}} \circ \gamma_{{n}_{k}}|_{[0, k-1]}) \leq M(\delta)
	\end{equation}
	for some constant $M(\delta)$ depending only on $\delta$. Pick $t_{n_{k}}\in [0, \infty)$ such that $d_{n_k}(x_{n_{k}}, \gamma_{n_{k}}(t_{n_{k}})) \leq M$. We claim that $(t_{n_{k}})_{k}$ has a uniform bound from above. In fact, we have
	\begin{equation}\label{stab1}
	d_{n_k}(p_{n_{k}}, x_{n_{k}})\leq d(p, f^{1/k}_{n_{k}}(x_{n_{k}}))+1/k \leq d(p, x)+ d(x, f^{1/k}_{n_{k}}(x_{n_{k}}))+1/k < C(x)
	\end{equation}
	where $C(x)=d(p, x)+2$. Therefore $x_{n_{k}} \in B_{d_{n_k}}(p_{n_{k}}, C(x))$ and we conclude that
	\[
	t_{n_{k}}=d_{n_k}(p_{n_{k}}, \gamma_{n_{k}}(t_{n_{k}})) \leq d_{n_k}(p_{n_{k}}, x_{n_{k}}) + d_{n_k}(x_{n_{k}}, \gamma_{n_{k}}(t_{n_{k}})) \leq  C(x)+M.
	\]
	Therefore, for $k$ large enough, $\gamma_{n_{k}}(t_{n_{k}}) \in B_{d_{n_k}}(p_{n_{k}}, k)$, $t_{n_{k}}\leq k-1$, and we get
	\begin{equation}\label{roughly starlike 1}
		d(f^{1/k}_{n_{k}}(x_{n_{k}}), f^{1/k}_{n_{k}}(\gamma_{n_{k}}(t_{n_{k}}))) \leq d_{n_k}(x_{n_{k}}, \gamma_{n_{k}}(t_{n_{k}}))+1/k\leq M+1/k.
	\end{equation}
	 From the stability consequence \eqref{stab}, there exists $s_{n_{k}} \in [0, l_d(\tilde{\gamma}_{n_k})]$ such that 
	\begin{equation}\label{roughly starlike 2}
		d(\tilde{\gamma}_{n_{k}}(s_{n_{k}}), f^{1/k}_{n_{k}}(\gamma_{n_{k}}(t_{n_{k}}))) \leq M(\delta).
	\end{equation}
	 Combining \eqref{roughly starlike 1} and \eqref{roughly starlike 2}, we have
	\[
	d(f^{1/k}_{n_{k}}(x_{n_{k}}), \tilde{\gamma}_{n_{k}})\leq d(f^{1/k}_{n_{k}}(x_{n_{k}}), \tilde{\gamma}_{n_{k}}(s_{n_k})) \leq M+M(\delta)+1/k.
	\]
	Moreover, since $d(x, f^{1/k}_{n_{k}}(x_{n_{k}}))< 1/k$, note that 
	\begin{equation}\label{roughly starlike 3}
		d(x, \tilde{\gamma}_{n_{k}}) \leq d(x, \tilde{\gamma}_{n_{k}}(s_{n_k})) \leq M+M(\delta)+2/k.
	\end{equation}
	%From here, we follow the standard argument involving Arzela-Ascoli theorem and properness of the space. Since 
	%\[
	%d(p, f^{1/k}_{n_{k}}(k-1)) \geq d(p_{n_{k}}, \gamma_{n_{k}}(k-1))-1/k=k-1-1/k
	%\]
	%which goes to $\infty$ as $k$ goes to $\infty$, by taking a subsequence, we may assume that $d(p, f^{1/k}_{n_{k}}(k-1))$ is monotone increasing approaching to $\infty$. Consider $\{\tilde{\gamma}_{n_{k}}|_{[0, 1]}\}_{k} \subseteq C([0, 1]: X)$. Equicontinuity for this sequence follows from the fact that each $\tilde{\gamma}_{n_{k}}$ is geodesic. And properness of the space assures $\{\tilde{\gamma}_{n_{k}}(t)\}$ is precompact for all $ t \in [0,1]$. Therefore by Aezela-Ascoli theorem, we can take a subsequence (still denoted by $\tilde{\gamma}_{n_{k}}$) convergent to $\gamma^1 \in C([0, 1]: X)$ uniformly on $[0, 1]$. Repeating this for the obtained susequence and by diagonal argument, we get a subsequence (still denoted by $\tilde{\gamma}_{n_{k}}$) such that $\tilde{\gamma}_{n_{k}}|_{[0, l]} \to \gamma^l$ uniformly on $[0, l]$ for all $ l \in \mathbb{N}$. For $t \in \mathbb{R}$, set 
	%\[
	%\gamma(t):= \gamma^l(t) = \lim_{k \to \infty} \tilde{\gamma}_{n_{k}}(t)
	%\]
	%where $t<l$. 
	By an Arzela-Ascoli type argument, there exists a geodesic ray $\gamma$ and a further subsequence of $\tilde{\gamma}_{n_{k}}$ (still denoted by $\tilde{\gamma}_{n_{k}}$) such that $\tilde{\gamma}_{n_{k}} \to \gamma$ locally uniformly as $k \to \infty$. From the estimates \eqref{stab1} and \eqref{roughly starlike 2}, we have 
	\[
	\begin{split}
	s_{n_{k}}=d(p, \tilde{\gamma}_{n_{k}}(s_{n_{k}})) &\leq d(p, f^{1/k}_{n_{k}}(\gamma(t_{n_{k}})))+d(f^{1/k}_{n_{k}}(\gamma(t_{n_{k}})), \tilde{\gamma}_{n_{k}}(s_{n_{k}}))\\
	 &\leq C(x)+M+2/k+ M(\delta).
	 \end{split}
	\]
	Hence $(s_{n_{k}})_{k}$ is bounded. By taking a subsequence again if needed, we may assume that $s_{n_{k}}$ is convergent to some $s \in \mathbb{R}$. Finally, from the uniform convergence of $\tilde{\gamma}_{n_{k}}|_{[0, l]}$ where $l> s$ and the estimate \eqref{roughly starlike 3}, we have 
	\[
	d(x, \gamma(s))=\lim_{k\to \infty} d(x , \tilde{\gamma}_{n_{k}}(s_{n_{k}}))\leq M+M(\delta).
	\]
\end{proof}

 \begin{proof}[Proof of Theorem \ref{first theorem}]
 	Combine Remark~\ref{limit is fine}, Proposition~\ref{GH-limit}, and Proposition~\ref{roughly starlike limit}.
 \end{proof}

\section{Stability of uniformized spaces, and boundaries}

In this section, we prove Theorem \ref{second theorem}. Recall again that we fix the constant $\epsilon_{0}(\delta)>0$ from Remark~\ref{epsilon constraint}. Also, unless otherwise stated, all curves in a metric space $(X, d)$ are assumed to be parametrized by arclength with respect to $d$.

%By Theorem \ref{first theorem}, together with Proposition \ref{compactness}, we get the following compactness theorem.
%\begin{cor}
%	Let $(X_n, d_n, p_n)_n$ be a sequence of pointed proper geodesic $M$-roughly starlike $\delta$-Gromov hyperbolic spaces. Suppose $(X_n, d_n, p_n)_{n}$ is pointed uniformly bounded. Then there exists a subsequence $(X_{n_{k}}, d_{n_k}, p_{n_{k}})_{k}$ pointed Gromov-Hausdorff convergent to a proper geodesic $\tilde{M}$-roughly starlike $\delta$-Gromov hyperbolic space $(X, d, p)$.
%\end{cor}

In order to prove Theorem \ref{second theorem}, we need the following lemmas.

\begin{lemma}\label{bhk}
 	Let $(X, d, p)$ be a pointed proper geodesic $\delta$-Gromov hyperbolic space and $0<\epsilon \leq \epsilon_{0}(\delta):=\frac{1}{14L}$ be a fixed constant where
 	\[
 	L :=6(1+24\delta)(2+40\delta)+48\delta+3.
 	\] 
 	Then for each $R>0$ and every pair of points $x$ and $y$ in $\bar{B}_d(p, R)$, there exists a curve $\gamma$ that is geodesic with respect to $d_{\epsilon}$ such that $\gamma \subseteq \bar{B}_d(p, 3R+L)$.
 \end{lemma}
 \begin{rmk}
 	Notice that whether the uniformized space $(X^{\epsilon}, d_{\epsilon})$ is geodesic or not is nontrivial since we do not have access to the Hopf-Rinow theorem. On the other hand by Definition \ref{uniformaization}, it is obvious that $(X^{\epsilon}, d_{\epsilon})$ is a length space. What this proposition tells us is that for $0<\epsilon \leq \epsilon_{0}(\delta)$, $(X^{\epsilon}, d_{\epsilon})$ is a geodesic metric space with geodesic curves that do not wander too far.
 \end{rmk}
 \begin{proof}
 	 	Given $x, y \in \bar{B}_d(p, R)$, we first construct a geodesic curve with respect to $d_{\epsilon}$. Take a curve $\tilde{\gamma}_k$ from $x$ to $y$ satisfying
 	 	\[
 	 	l_{d_{\epsilon}}(\tilde{\gamma}_k)\leq d_{\epsilon}(x, y)+\frac{1}{k}
 	 	\]
 	 	for each $k \in \mathbb{N}$. By \cite[Lemma 5.7]{BHK}, we can take another curve $\hat{\gamma}_k$ with the same endpoints as $\tilde{\gamma}_k$ with $l_{d_{\epsilon}}(\hat{\gamma}_k)\leq l_{d_{\epsilon}}(\tilde{\gamma}_k)$ and
 	 	\[
 	 	l_{d}(\hat{\gamma}_k|_{(s, t)}) \leq 3d(\hat{\gamma}_k(s),\hat{\gamma}_k(t))+1
 	 	\]
 	 	whenever $0 \leq s \leq t \leq l_{d}(\hat{\gamma}_k)$ are such that $d(\hat{\gamma}_k(s), \hat{\gamma}_k(t))\leq \frac{1}{12\epsilon}$.
 	 	Since the curves $\hat{\gamma}_k$ satisfy the assumptions of \cite[Lemma 5.21]{BHK} by the choice of $\epsilon_{0}(\delta)$, the curve $\hat{\gamma}_k$ belongs to the $L$-neighborhood of each geodesic $[x, y]$ with respect to the original metric $d$. Since $[x, y] \subseteq \bar{B}_d(p, 3R)$, we conclude that $\hat{\gamma}_k \subseteq ([x, y])_{L} \subseteq \bar{B}_d(p, 3R+L)$. Since the diameter of the uniformized space $X^{\epsilon}$ is at most $\frac{2}{\epsilon}$, we get
\[
2+\frac{2}{\epsilon}\geq 1+1/k+l_d(\tilde{\gamma}_{k})\geq \int_{0}^{l_{d}(\hat{\gamma}_k)}e^{-\epsilon d(p, \hat{\gamma}_k(t))}\, dt \geq e^{-\epsilon (3R+L)}l_{d}(\hat{\gamma}_k).
\] 	
Therefore,
\begin{equation}\label{(2)}
 l_{d}(\hat{\gamma}_k)\leq \Big(2+\frac{2}{\epsilon}\Big)e^{\epsilon (3R+L)}=:M
\end{equation} 
for every $k \in \mathbb{N}$. Set $\alpha_k(t) := l_d(\hat{\gamma}_k)t$ for each $t \in [0, 1]$ and consider $\gamma_k := \hat{\gamma}_k \circ \alpha_k : [0, 1] \to X$. Then we have, for each $t, \tilde{t} \in [0, 1]$ with $t \leq \tilde{t}$,
\[
	d(\gamma_k(t), \gamma_k(\tilde{t})) \leq l_d(\gamma_k|_{[t, \tilde{t}]})\\
	\leq l_d(\hat{\gamma}_k|_{[\alpha_k(t), \alpha_k(\tilde{t})]})\\
	\leq l_d(\hat{\gamma}_k)|\tilde{t}-t|\\
	\leq M|\tilde{t}-t|,
\]
which implies that $(\gamma_k)_k \subseteq C([0, 1] : X)$ is equicontinuous. Hence by the Arzela-Ascoli theorem, we have a curve $\gamma : [0, 1] \to X$ and a subsequence $\gamma_{k}$ (still denoted by $\gamma_k$) such that $\gamma_k \to \gamma$ uniformly with respect to the metric $d$. Since the identity map $Id : (X, d) \to (X, d_{\epsilon})$ is homeomorphic by \cite[Appendix]{BHK},  for every partition $t_{1}=0< \cdots < t_{N}=1$, 
\[
	\sum\limits_{i = 1}^{N-1} d_{\epsilon}(\gamma(t_{i}, \gamma(t_{i+1}))) \leq \lim\limits_{k \to \infty}\sum\limits_{i=1}^{N-1}d_{\epsilon}(\gamma_k(t_i), \gamma_k(t_{i+1}))\\
	\leq \lim\limits_{k \to \infty}l_{d_\epsilon}(\gamma_n)\\
	\leq d_{\epsilon}(x, y),
\]
 which implies that the curve $\gamma$ is geodesic with respect to $d_{\epsilon}$ which stays in $[x, y]_{L}$. This completes the proof.
 \end{proof}

 \begin{lemma}\label{key lemma of ball conv 1}
 	Let $(X, d, p)$ be a pointed proper geodesic $\delta$-Gromov hyperbolic space and $R \geq 1$, $0<\epsilon\leq \epsilon_0(\delta)$, and $\delta'>0$ be given. Set $T:=4+L$ where $L$ is the constant in Lemma \ref{bhk}. For every pair of points  $x, y \in \bar{B}_d(p, R)$ with $d(x, y)\geq \delta'/2$, let $\gamma \subseteq B_d(p, TR)$ be a geodesic curve from $x$ to $y$ with respect to $d_{\epsilon}$ taken from Lemma \ref{bhk}. Then  we can take subcurves $(\gamma_i)_{i=1}^{N}$ of $\gamma$ such that
 	\[
 	 \gamma=\sum_{i=1}^{N}\gamma_{i}, \ \ \ \delta'/2 \leq l_d(\gamma_i) < \delta' \ \ \  \text{and} \ \ \ N \leq \frac{4}{\epsilon \delta'}e^{\epsilon T R}.
 	 \]
 \end{lemma}
 \begin{proof}
 	 By Lemma~\ref{bhk}, for every pair of points $x, y \in \bar{B}_d(p, R)$, we can always find a geodesic curve $\gamma$ with respect to $d_{\epsilon}$ such that $\gamma \subseteq B_d(p, TR)$. Since the diameter of the uniformized space $X^{\epsilon}$ is at most $\frac{2}{\epsilon}$, we get
\[
\frac{2}{\epsilon}\geq \int_{0}^{l_{d}(\gamma)}e^{-\epsilon d(p, \gamma_n(t))}\, dt \geq e^{-\epsilon T R}l_{d}(\gamma).
\] 	
Therefore, we have
\begin{equation}\label{(2)}
 l_{d}(\gamma)\leq \frac{2e^{\epsilon T R}}{\epsilon}.
\end{equation}
Now split this curve $\gamma$ into $N$ subcurves $(\gamma_{i})_{i=1}^{N}$ so that $\delta'/2 \leq t_{i}:=l(\gamma_{i})<\delta'$ for $1\leq i \leq N$. Since 
\[
\delta'/2\times N \leq \sum\limits_{i=1}^{N} l_{d}(\gamma_{i})= l_{d}(\gamma)  \leq  \frac{2e^{\epsilon T R}}{\epsilon},
\]  
 we have $N \leq \frac{4}{\epsilon \delta'}e^{\epsilon T R}$.
 \end{proof}
 
 \begin{rmk}\label{choice of delta'}
 	In order to prove the next lemma, we note that for each $\epsilon>0$, the function $t \mapsto e^{-\epsilon t}$ is uniformly continuous on $[0, \infty)$, i.e., for any $\tilde{\epsilon}>0$, there exists $0<\delta' \leq \tilde{\epsilon}\wedge 1$ such that
 	\[
 	|e^{-\epsilon t}-e^{-\epsilon t'}|<\tilde{\epsilon}
 	\]
 	holds for every $t, t' \in [0, \infty)$ with $|t-t'| \leq 4\delta'$. This choice of $\delta'$ plays an important role in proving the convergence of uniformized spaces.
 \end{rmk}

 \begin{lemma}\label{key lemma of ball conv 2}
 	Let $(X, d, p)$ be a pointed proper geodesic $\delta$-Gromov hyperbolic space, $R \geq 1$ and $0<\epsilon\leq \epsilon_0(\delta)$. Let $T=4+L$ be the constant from Lemma~\ref{key lemma of ball conv 1}. For any $\tilde{\epsilon}>0$, let $\delta'>0$ be as in Remark~\ref{choice of delta'} corresponding to $\tilde{\epsilon}>0$.  Then for every $x, y \in \bar{B}_d(p, R)$ with $d(x, y)\geq \delta'/2$, we have
 	\begin{equation}\label{(3)}
 	d_{\epsilon}(x, y) = \int_{0}^{l_{d}(\gamma)}e^{-\epsilon d(p, \gamma(t))}\, dt \geq \sum\limits_{i=1}^{N}e^{-\epsilon d(p, \gamma_{i}(t_i))}l_{d}(\gamma_{i})-\frac{2\tilde{\epsilon}}{\epsilon}e^{\epsilon T R}\notag
 		\end{equation}
 	where $\gamma \subseteq B_d(p, TR)$ and subcurves $(\gamma_i)_{i=1}^{N}$ of $\gamma$ are as in the statement of Lemma~\ref{key lemma of ball conv 1} corresponding to the constant $\delta'$ and $t_i:=l_d(\gamma_i)$ for $i=1, \cdots, N$.
 \end{lemma}
 \begin{proof}
 	 By Lemma~\ref{key lemma of ball conv 1}, there exist a geodesic curve $\gamma \subseteq B_d(p, TR)$ from $x$ to $y$ with respect to $d_{\epsilon}$  and subcurves $(\gamma_i)_{i=1}^{N}$ of $\gamma$ such that
 	\begin{equation}
 	 \delta'/2 \leq l_d(\gamma_i) < \delta' \ \ \  \text{and} \ \ \ N \leq \frac{4}{\epsilon \delta'}e^{\epsilon T R}.\notag
 	 \end{equation}
 	 Then we have, by \eqref{(2)} and the uniform continuity of $e^{-\epsilon t}$,
 		\begin{align}\label{a}
 		\int_{0}^{l_{d}(\gamma)}e^{-\epsilon d(p, \gamma(t))}\, dt &= \sum\limits_{i=1}^{N}\int_{0}^{l_{d}(\gamma_{i})}e^{-\epsilon d(p, \gamma_{i}(t))} \, dt\notag \\
 		&\geq \sum\limits_{i=1}^{N}\int_{0}^{l_{d}(\gamma_{i})}(e^{-\epsilon d(p, \gamma_{i}(t_i))}-\tilde{\epsilon}) \, dt\notag \\
 		&\geq \sum\limits_{i=1}^{N}e^{-\epsilon d(p, \gamma_{i}(t_i))}l_{d}(\gamma_{i})-\tilde{\epsilon} l_{d}(\gamma)\notag \\
 		&\geq \sum\limits_{i=1}^{N}e^{-\epsilon d(p, \gamma_{i}(t_i))}l_{d}(\gamma_{i})-\frac{2\tilde{\epsilon}}{\epsilon}e^{\epsilon T R}. \notag
 	\end{align}
 	This completes the proof.
 \end{proof}

 \begin{lemma}\label{key lemma of ball conv 3}
 		Let $(X_n, d_n, p_n)_n$ and $(X, d, p)$ be pointed proper geodesic $\delta$-Gromov hyperbolic spaces. Suppose that $(X_n, d_n, p_n)_n$ is pointed Gromov-Hausdorff convergent to $(X, d, p)$. Fix $0<\epsilon\leq \epsilon_0(\delta)$ and let $T:=3+L$ be the constant from Lemma~\ref{key lemma of ball conv 1}. Let $R\geq 1$ be given. For any $\tilde{\epsilon}>0$, let $\delta'>0$ be the constant given in Remark~\ref{choice of delta'}. Then there exists $N:=N_{T, R, \tilde{\epsilon}} \in \mathbb{N}$ such that for $n \geq N$, we have
 		\[
  |d_{n, \epsilon}(x_n, y_n)-d_{\epsilon}(f_{n}^{\delta'^2}(x_n), f_{n}^{\delta'^2}(y_n))| \leq S(\epsilon, T, R)\tilde{\epsilon}
  \]
for every pair of points $x_n, y_n \in \bar{B}_{d_n}(p_n, R)$ where $f_{n}^{\delta'^2} : B_{d_n}(p_n, TR) \to X$ is a map as in Definition~\ref{p-GH} and $S(\epsilon, T, R)\tilde{\epsilon} \to 0$ as $\tilde{\epsilon} \to 0$.
 	\end{lemma}
 	\begin{proof}
 		Let $\tilde{\epsilon}>0$ and $R \geq 1$ be given. We also fix the constant $\delta'>0$ in Remark~\ref{choice of delta'}.  Since $(X_n, d_n, p_n)_n$ pointed Gromov-Hausdorff converges to $(X, d, p)$, there exists $N:=N_{T, R, \tilde{\epsilon}} \in \mathbb{N}$ such that for $n \geq N$, there exists a map $f_{n}^{\delta'^2} : B_{d_n}(p_n, TR) \to X$ satisfying the properties in Definition~\ref{p-GH}. We may assume that $f_{n}^{\delta'^2}(\bar{B}_{d_n}(p_n, R))\subseteq \bar{B}_d(p, R)$ since $X_n$ are all geodesic, see Remark~\ref{p-GH def remark}.
 		For given $x_n, y_n \in \bar{B}_{d_n}(p_n, R)$, first note that if $d_n(x_n, y_n)< \delta'/2$, then it is always true that 
 		\[
 		|d_{n, \epsilon}(x_n, y_n)-d_{\epsilon}(f_{n}^{\delta'^2}(x_n), f_{n}^{\delta'^2}(y_n))| \leq 3\tilde{\epsilon}
 		\] 
 		since $d_{\epsilon}(x, y) \leq d(x, y)$ for all $x, y \in X$ and $d_{n, \epsilon}(x_n, y_n) \leq d_n(x_n, y_n)$ for all $x_n, y_n \in X_n$. Hence we may assume that $d_n(x_n, y_n)\geq \delta'/2$. By Lemma \ref{key lemma of ball conv 2}, we can find a geodesic curve $\gamma_n \subseteq B_{d_n}(p_n, TR)$ with respect to $d_{n, \epsilon}$, and subcurves $(\gamma_{n, i})_{i=1}^{N}$ of $\gamma_n$ such that  
 	\begin{equation}\label{(3)}
 		d_{n, \epsilon}(x_n, y_n) = \int_{0}^{l_{d_n}(\gamma_n)}e^{-\epsilon d_n(p_n, \gamma_n(t))}\, dt \geq \sum\limits_{i=1}^{N}e^{-\epsilon d_n(p_n, \gamma_{n, i}(t_i))}l_{d_n}(\gamma_{n, i})-\frac{2\tilde{\epsilon}}{\epsilon}e^{\epsilon T R},
 	\end{equation}
 	where $t_i:=l_{d_n}(\gamma_{n, i})$ for $i=1, \cdots, N$.
 	Note that by the property of the map $f_{n}^{\delta'^2}$ in Definition \ref{p-GH}, for $(f_{n}^{\delta'^2}(\gamma_{n, i}(t_i)))_{i=1}^{N}$,
 	\[
 	\begin{split}
 		|d_n(\gamma_{n, i}(t_i), \gamma_{n, i+1}(t_{i+1}))- d(f_{n}^{\delta'^2}(\gamma_{n, i}(t_i)), f_{n}^{\delta'^2}(\gamma_{n, i+1}(t_{i+1})))|< \delta'^2
 	\end{split}
 	\]
 	and
 	\[
 	\begin{split}
 		|d(p, f_{n}^{\delta'^2}(\gamma_{n, i}(t_i)))- d_n(p_n, \gamma_{n, i}(t_i))| \leq \delta'^2
 	\end{split}
 	\]
 	 for $1 \leq i \leq N$. Let $\tilde{\gamma}_{1}$ be a geodesic curve  from $f_{n}^{\delta'^2}(x_n)$ to $f_{n}^{\delta'^2}(\gamma_{n, 1}(t_1))$ and $\tilde{\gamma}_{i}$ be a geodesic curve from $f_{n}^{\delta'^2}(\gamma_{n, i-1}(t_{i-1}))$ to $f_{n}^{\delta'^2}(\gamma_{n, i}(t_{i}))$ for $2 \leq i \leq N$ with respect to $d$.  Then for $2\leq i \leq N$, we get
 	\[
 		t_i':=l_{d}(\tilde{\gamma}_{i}) = d(f_{n}^{\delta'^2}(\gamma_{n, i-1}(t_{i-1})), f_{n}^{\delta'^2}(\gamma_{n, i}(t_i))) \leq d_n(\gamma_{n, i-1}(t_{i-1}), \gamma_{n, i}(t_{i}))
 		+\delta'^2 \leq 2\delta'.
 	\]
 	It is also clear that $t_i':=l_d(\tilde{\gamma}_1) \leq 2\delta'$. Hence, from \eqref{(3)} and the uniform continuity of $e^{-\epsilon t}$, we have
 		\begin{align}\label{b}
 		\sum\limits_{i=1}^{N}e^{-\epsilon d(\tilde{\gamma}_{i}(t_i'), p)}l_{d}(\tilde{\gamma}_{i}) 
 		&\leq \sum\limits_{i=1}^{N}(e^{-\epsilon d_n(p_n, \gamma_{n, i}(t_i))}+\tilde{\epsilon})l_{d}(\tilde{\gamma}_{i})\notag \\
 		&\leq \sum\limits_{i=1}^{N}(e^{-\epsilon d_n(p_n, \gamma_{n, i}(t_i))}+\tilde{\epsilon})(l_{d_n}(\gamma_{n, i})+\delta'^2)\notag \\
 		&\leq \sum\limits_{i=1}^{N}e^{-\epsilon d_n(p_n, \gamma_{n, i}(t_i))}l_{d_n}(\gamma_{n, i})+\sum\limits_{i=1}^{N}e^{-\epsilon d_n(p_n, \gamma_{n, i}(t_i))}\delta'^2\notag \\
 		&\ \ +\tilde{\epsilon} l_{d_n}(\gamma_{n})+N\tilde{\epsilon}\delta'^2 \notag \\
 		&\leq \int_{0}^{l_{d_n}(\gamma_n)}e^{-\epsilon d_n(p_n, \gamma_n(t))}\, dt +\frac{2\tilde{\epsilon}}{\epsilon}e^{\epsilon T R}+N \delta'^2\notag \\
 		&\ \ +\tilde{\epsilon} l_{d_n}(\gamma_{n})+N\tilde{\epsilon}\delta'^2 \notag \\
 		&\leq  \int_{0}^{l_{d_n}(\gamma_n)}e^{-\epsilon d_n(p_n, \gamma_n(t))}\, dt +\frac{2\tilde{\epsilon}}{\epsilon}e^{\epsilon T R}+\Big(\frac{4}{\epsilon \delta'}e^{\epsilon T R}\Big)\delta'^2 \notag \\
 		&\ \ +\frac{\tilde{\epsilon}}{\epsilon}e^{\epsilon T R} +\Big(\frac{4}{\epsilon \delta'}e^{\epsilon T R}\Big)\tilde{\epsilon}\delta'^2\notag \\
 		&= \int_{0}^{l_{d_n}(\gamma_n)}e^{-\epsilon d_n(p_n, \gamma_n(t))}\, dt+C(\epsilon, T, R)\tilde{\epsilon}
 		\end{align}
 	where 
 	\[
 	C(\epsilon, T, R):=\frac{11}{\epsilon}e^{\epsilon T R}.
 	\]
 	Also, noting that for each $t \in [0, l(\tilde{\gamma}_{i})]$ ($1\leq i \leq N$),
 	\[
 	\begin{split}
 	d(\tilde{\gamma}_{i}(t), \tilde{\gamma}_{i}(t_{i})) &\leq l_d(\tilde{\gamma}_{i}) \leq 2 \delta',\\
 	\end{split}
 	\]
 	we get 
 		\begin{align}\label{c}
 		\sum\limits_{i=1}^{N}e^{-\epsilon d(p, \tilde{\gamma}_{i}(t_i'))}l_d(\tilde{\gamma}_{i}) &\geq \sum\limits_{i=1}^{N}\int_{0}^{l_d(\tilde{\gamma}_{i})}(e^{-\epsilon d(p, \tilde{\gamma}_{i}(t))}-\tilde{\epsilon}) \, dt \notag \\
 		&\geq \sum\limits_{i=1}^{N}\int_{0}^{l_d(\tilde{\gamma}_{i})}e^{-\epsilon d(p, \tilde{\gamma}_{i}(t))}\, dt -(N\tilde{\epsilon})(2\delta') \notag \\
 		&\geq d_{\epsilon}(f_{n}^{\delta'^2}(x_n), f_{n}^{\delta'^2}(y_n))-2\tilde{\epsilon}\delta'\Big(\frac{4}{\epsilon \delta'}e^{\epsilon T R}\Big) \notag \\
 		&\geq d_{\epsilon}(f_{n}^{\delta'^2}(x_n), f_{n}^{\delta'^2}(y_n))-2\tilde{\epsilon}\Big(\frac{4}{\epsilon}e^{\epsilon T R}\Big)\notag \\
 		&\geq d_{\epsilon}(f_{n}^{\delta'^2}(x_n), f_{n}^{\delta'^2}(y_n))-K(\epsilon, T, R)\tilde{\epsilon}
 	\end{align}
  	where $K(\epsilon, T, R):=\frac{8}{\epsilon}e^{\epsilon T R}$ and we used the estimate for $N$ in Lemma \ref{key lemma of ball conv 1}.
  	Combining the inequalities \eqref{b} and \eqref{c} , we get  
  	\[
  	d_{\epsilon}(f_{n}^{\delta'^2}(x_n), f_{n}^{\delta'^2}(y_n))\leq d_{n, \epsilon}(x_n, y_n)+ K(\epsilon, T, R)\tilde{\epsilon} + C(\epsilon, T, R)\tilde{\epsilon}.
  	\]
Hence we conclude that
  	\[
  	 d_{\epsilon}(f_{n}^{\delta'^2}(x_n), f_{n}^{\delta'^2}(y_n))-d_{n, \epsilon}(x_n, y_n)\leq K(\epsilon, T, R)\tilde{\epsilon} + C(\epsilon, T, R)\tilde{\epsilon}.
  	\]

  	 Next by picking $x_{n}$ and $y_n$ from $\bar{B}_{d_n}(p_n, R)$ and a geodesic curve $\gamma$ between $f_{n}^{\delta'^2}(x_n)$ and $f_{n}^{\delta'^2}(y_n)$ with respect to the uniformization metric $d_{\epsilon}$, and doing the same argument for $\gamma$ instead of $\gamma_n$, we also have,
  	\[
  	d_{n, \epsilon}(x_n, y_n) - d_{\epsilon}(f_{n}^{\delta'^2}(x_n), f_{n}^{\delta'^2}(y_n))\leq \tilde{K}(\epsilon, T, R)\tilde{\epsilon}  + \tilde{C}(\epsilon, T, R)\tilde{\epsilon}.
  	\]
  	where
  	\[
  	\tilde{C}(\epsilon, T, R)\tilde{\epsilon}:= \frac{33}{\epsilon}e^{\epsilon T R}
  	\]
  	 and
  	  \[
  	  \tilde{K}(\epsilon, T, R)\tilde{\epsilon}:=\frac{24}{\epsilon}e^{\epsilon T R}.
  	\] 
    Therefore, we conclude that
  \[
  |d_{n, \epsilon}(x_n, y_n)-d_{\epsilon}(f_{n}^{\delta'^2}(x_n), f_{n}^{\delta'^2}(y_n))| \leq S(\epsilon, T, R)\tilde{\epsilon}
  \]
  for every $x_n, y_n \in \bar{B}_{d_n}(p_n, R)$ where
  \[
  S(\epsilon, T, R):=2(K(\epsilon, T, R)\vee \tilde{K}(\epsilon, T, R) \vee C(\epsilon, T, R) \vee \tilde{C}(\epsilon, T, R)).
  \]
 	\end{proof}

 \begin{prop}\label{ball-conv}
 	Let $(X_n, d_n, p_n)_n$ and $(X, d, p)$ be pointed proper geodesic $\delta$-Gromov hyperbolic spaces. Suppose that $(X_n, d_n, p_n)_n$ is pointed Gromov-Hausdorff convergent to $(X, d, p)$ and $R \geq 1$ and $0< \epsilon \leq \epsilon_{0}(\delta)$ are given. Then 
 	\[
 	\lim\limits_{n\to \infty}d_{GH}((\bar{B}_{d_n}(p_n, R), d_{n, \epsilon}), (\bar{B}_d(p, R), d_{\epsilon}))=0.
 	\]
 \end{prop}
 \begin{proof}
 	Set $T:=4+L$ as in Lemma~\ref{key lemma of ball conv 1}. For any $\tilde{\epsilon}>0$, by Lemma~\ref{key lemma of ball conv 2} and Lemma~\ref{key lemma of ball conv 3}, there exist $\delta'>0$ and $N:=N_{T, R, \tilde{\epsilon}} \in \mathbb{N}$ such that for $n \geq N$, we have
 		\[
  |d_{n, \epsilon}(x_n, y_n)-d_{\epsilon}(f_{n}^{\delta'^2}(x_n), f_{n}^{\delta'^2}(y_n))| \leq S(\epsilon, T, R)\tilde{\epsilon}
  \]
for every pair of points $x_n, y_n \in \bar{B}_{d_n}(p_n, R)$ where $f_{n}^{\delta'^2} : B_{d_n}(p_n, TR) \to X$ is the map in Definition~\ref{p-GH} and $S(\epsilon, T, R)\tilde{\epsilon} \to 0$ as $\tilde{\epsilon} \to 0$. Recall again that $0<\delta' \leq \tilde{\epsilon}$ by Remark~\ref{choice of delta'}. We may assume again that $f_{n}^{\delta'^2}(\bar{B}_{d_n}(p_n, R)) \subseteq \bar{B}_d(p, R)$. 
From Property (3) in Definition~\ref{p-GH}, we know that, with respect to the metric $d$,
  \[
  \begin{split}
  	\bar{B}_d(p, R-\delta'^2) &\subseteq (f_{n}^{\delta'^2}(\bar{B}_{d_n}(p_n, R)))_{\delta'^2}\\
  	%&\subseteq (\phi(B(p_n, R)))_{2\delta'^2}
  \end{split}
  \]
  which tells us that 
  \[
  \bar{B}_d(p, R)\subseteq (f_{n}^{\delta'^2}(\bar{B}_{d_n}(p_n, R)))_{2\delta'^2}
  \]
  since the space is geodesic. This above inclusion holds with respect to $d_{\epsilon}$ as well, so that we get
  \[
  	\bar{B}_d(p, R) \subseteq (f_{n}^{\delta'^2}(\bar{B}_{d_n}(p_n, R)))_{2\delta'^2}
  \]
  with respect to $d_{\epsilon}$. Hence, the map $f_{n}^{\delta'^2} : \bar{B}_{d_n}(p_n, R) \to \bar{B}_{d}(p, R)$ is $M(\epsilon, T, R)\tilde{\epsilon}$-isometry where $M(\epsilon, T, R):=S(\epsilon, T, R)\vee 2$. Hence we conclude that
  \[
  \limsup\limits_{n\to \infty}d_{GH}((\bar{B}_d(p_n, R), d_{n, \epsilon}), (\bar{B}_d(p, R), d_{\epsilon}))\leq M(\epsilon, T, R)\tilde{\epsilon}.
  \]
  Since $\tilde{\epsilon}$ is arbitrary, the proof is complete.
 \end{proof}

 \begin{rmk}\label{ball remark}
	From the proof of this proposition, we know that for fixed $R \geq 1$ and $\tilde{\epsilon}>0$, there exists $N \in \mathbb{N}$ such that for each $n \geq N$, there exists a map $\phi_n : \bar{B}_{d_n}(p_n, R) \to \bar{B}_d(p, R)$ which is an $\tilde{\epsilon}$-isometry with respect to both the uniformization metric and the original metric such that $\phi_n(p_n)=p$.
\end{rmk}

\begin{thm}[Theorem \ref{second theorem}]\label{boundary conv}
	Let $(X_n, d_n, p_n)_n$ be pointed proper geodesic $M$-roughly starlike $\delta$-Gromov hyperbolic spaces. Suppose $(X_n, d_n, p_n)_n$ is pointed Gromov-Hausdorff convergent to $(X, d, p)$. Then for a fixed $0 < \epsilon \leq \epsilon_0(\delta)$, the sequence of uniformized spaces $(X_n^{\epsilon}, d_{n, \epsilon})_n$ Gromov-Hausdorff converges to $(X^{\epsilon}, d_{\epsilon})$ with boundary.
\end{thm}
	\begin{proof}
	Let $\tilde{\epsilon}>0$ be given. Choose $R>0$ such that
	\begin{equation}\label{choice of R}
		\frac{e^{-\epsilon(R-2\tilde{\epsilon}-M)}}{\epsilon} \leq \tilde{\epsilon}.
	\end{equation}
	By Remark \ref{ball remark}, there exists $N \in \mathbb{N}$ such that for each $n \geq N$, there exists an $\tilde{\epsilon}$-isometry $\phi_n : \bar{B}_{d_n}(p_n, R) \to \bar{B}_d(p, R)$  with respect to both the uniformization metric and the original metric such that $\phi_n(p_n)=p$. Let $x_n \in \overline{X^{\epsilon}_n} \setminus \bar{B}_{d_n}(p_n, R)$. Then there exists a geodesic curve (or ray) $\gamma_n$ from $p_n$ to $x_n$ with respect to $d_n$, see Remark \ref{Gromov identification}. Set $\tilde{x}_n:=\gamma_n(R) \in \bar{B}_{d_n}(p_n, R)$. We define a map $\Phi_n : \overline{X^{\epsilon}_{n}} \to \overline{X^{\epsilon}}$ by
	 \[
	  \Phi_n(x_n):=
     \begin{cases}
       \phi_n(x_n) &\quad\text{if $x_n \in \bar{B}(p_n, R)$}\\
       \phi_n(\tilde{x}_n) &\quad\text{if $x_n \in \overline{X^{\epsilon}_n} \setminus \bar{B}_{d_n}(p_n, R)$}. \\
     \end{cases}
     \]
     By Proposition \ref{isometry with boundary}, it is sufficient to prove that the map $\Phi_n$ is a $5\tilde{\epsilon}$-isometry with boundary. Note that for every $x_n \in \overline{X^{\epsilon}_n} \setminus \bar{B}_{d_n}(p_n, R)$,
	\begin{equation}\label{modification 1}
		d_{n, \epsilon}(\tilde{x}_n, x_n) \leq \int_{R}^{l_{d_n}(\gamma)}e^{-\epsilon d(p_n, \gamma_n(t))}\, dt \leq \int_{R}^{\infty}e^{-\epsilon t}\, dt  \leq \tilde{\epsilon}
		%= \frac{e^{-\epsilon R}}{\epsilon}
	\end{equation}
	where $\gamma_n$ is a geodesic curve(or ray) from $p_n$ to $x_n$ with respect to the metric $d$. Hence, for every pair of points $x_n, x_n' \in \overline{X^{\epsilon}_n}$, we have
	\begin{align}\label{gb3}
		|d_{\epsilon}(\Phi_n(x_n), \Phi_n(x_n'))-d_{n, \epsilon}(x_n, x_n')| &\leq |d_{\epsilon}(\Phi_n(x_n), \Phi_n(x_n'))-d_{n, \epsilon}(\tilde{x}_n, \tilde{x}_n'))| \notag \\
		&\ \ \ +|d_{n, \epsilon}(\tilde{x}_n, \tilde{x}_n'))-d_{n, \epsilon}(x_n, x_n')| \notag \\
		&\leq \tilde{\epsilon} + |d_{n, \epsilon}(\tilde{x}_n, \tilde{x}_n')) - d_{n, \epsilon}(\tilde{x}_n, x_n'))|\notag \\
		&\ \ \ +|d_{n, \epsilon}(\tilde{x}_n, x_n')-d_{n, \epsilon}(x_n, x_n')| \notag \\
		&\leq 3 \tilde{\epsilon}.\notag
	\end{align}
	where  we set $\tilde{x}_n:=x_n$ if $x_n \in \bar{B}_{d_n}(p_n, R)$ and used  \eqref{modification 1}. Therefore we obtain
	\begin{equation}
		\dis(\Phi_n) \leq 3\tilde{\epsilon} \notag
	\end{equation}
	with respect to the uniformization metric. We next prove that $X^{\epsilon} \subseteq (\Phi_n(X^{\epsilon}_n))_{2\tilde{\epsilon}}$. Let $x \in X^{\epsilon}$. We may assume that $x \in X^{\epsilon} \setminus \bar{B}_{d}(p, R)$ and take a geodesic curve $\gamma$ with respect to $d$ from $p$ to $x$, setting $\tilde{x}:=\gamma(R)$. By the same argument as in \eqref{modification 1}, we know that $d_{\epsilon}(x, \tilde{x}) \leq \tilde{\epsilon}$. Since $\phi_n : \bar{B}_{d_n}(p_n, R) \to \bar{B}_d(p, R)$  is $\tilde{\epsilon}$-isometry with respect to the uniformization metric, there exists $x_n \in \bar{B}_d(p, R)$ such that $d_{\epsilon}(\phi_n(x_n), \tilde{x}) \leq \tilde{\epsilon}$. Hence we obtain 
	\[
	d_{\epsilon}(x, \Phi_n(x_n)) =d_{\epsilon}(x, \phi_n(x_n)) \leq d_{\epsilon}(x, \tilde{x})+ d_{\epsilon}(\tilde{x}, \phi_n(x_n)) \leq 2\tilde{\epsilon}.
	\]
	We next claim that $\partial_{d_{\epsilon}} X^{\epsilon} \subseteq (\Phi(\partial_{d_{n, \epsilon}}X^{\epsilon}_{n}))_{5\tilde{\epsilon}}$. By Remark~\ref{Gromov identification}, for any $x \in \partial_{d_{\epsilon}} X^{\epsilon}$, we can take a geodesic ray $\gamma : [0, \infty) \to X$ with respect to the metric $d$ such that $d_{\epsilon}(x, \gamma(k)) \to 0$ as $k \to \infty$. Set $\tilde{x}:= \gamma(R)$. As in (\ref{modification 1}), we know that 
	\begin{equation}\label{GHB 3}
		d_{\epsilon}(x, \tilde{x}) \leq \int_{R}^{\infty}e^{-\epsilon t}\, dt \leq  \tilde{\epsilon}.
	\end{equation} 
	Since $\phi_n : \bar{B}_{d_n}(p_n, R) \to \bar{B}_d(p, R)$ is an $\tilde{\epsilon}$-isometry with respect to the uniformization metric, there exists $x_n \in \bar{B}_{d_n}(p_n, R)$ such that $d(\phi_n(x_n), \tilde{x}) \leq \tilde{\epsilon}$. The $M$-roughly starlike property of $X_n$ allows us to take a geodesic ray $\gamma_n : [0, \infty) \to X_n$ emanating from $p_n$  with respect to $d_n$ such that $\dist (x_n, \gamma_n) \leq M$. Pick $z_n \in \gamma_n$ such that $d(x_n, z_n)= \dist(x_n, \gamma_n)$. Then since
	\[
	d_n(p_n, x_n) \geq d(p, \phi_n(x_n))-\tilde{\epsilon} \geq d(p, \tilde{x})-d(\tilde{x}, \phi_n(x_n))-\tilde{\epsilon}\geq R-2\tilde{\epsilon},
	\]
	we get
	\begin{equation}\label{GHB 4}
		d_{n, \epsilon}(x_n, z_n) \leq \int_{0}^{l_d(\beta)} e^{-\epsilon d(p, \beta(t))}\, dt \leq e^{-\epsilon (R - 2\tilde{\epsilon})}\int_{0}^{l_d(\beta)} e^{\epsilon t}\, dt \leq \tilde{\epsilon}
	\end{equation}
	where $\beta$ is a geodesic curve from $x_n$ to $z_n$ and we employed the Harnack inequality (\ref{Harnack}). Also, noting that
	\begin{equation}
		d_n(p_n, z_n)\geq d_n(p_n, x_n)-d_n(x_n, z_n)\geq d(p, \phi_n(x_n))-\tilde{\epsilon}-M \geq R-2\tilde{\epsilon}-M, \notag
	\end{equation}
	we have 
	\begin{equation}\label{GHB 5}
 d_{n, \epsilon}(z_n, \gamma_n(R))\leq \int_{R-2\tilde{\epsilon}-M}^{\infty}e^{-\epsilon t} \, dt\leq \tilde{\epsilon}.
 \end{equation} 
	Set $y:=\Phi(a_n):=\Phi_n(\gamma_n(R))=\phi_n(\gamma_n(R))$ where $a_n \in \partial_{d_{n, \epsilon}}X^{\epsilon}_{n}$ is the unique point such that $\gamma_n(k) \to a_n$ as $k \to \infty$. %Since we know that $X$ is $\tilde{M}$-roughly starlike by Proposition \ref{}, there exists a geodesic ray $\gamma' : [0, \infty) \to X$ such that $\dist (y, \gamma') \leq \tilde{M}$.
	Then, combining (\ref{GHB 3}), (\ref{GHB 4}), and (\ref{GHB 5}), we have
	\begin{align}
		d_{\epsilon}(x, y) &\leq d_{\epsilon}(x, \tilde{x})+ d_{\epsilon}(\tilde{x}, \phi_n(x_n))+d_{\epsilon}(\phi_n(x_n), y)\notag \\
		&\leq d_{\epsilon}(x, \tilde{x})+ d(\tilde{x}, \phi_n(x_n))+d_{\epsilon}(\phi_n(x_n), y)\notag \\
		&\leq 3\tilde{\epsilon}+d_{n, \epsilon}(x_n, \gamma_n(R)) \notag \\
		&\leq 3\tilde{\epsilon}+d_{n, \epsilon}(x_n, z_n)+d_{n, \epsilon}(z_n, \gamma_n(R)) \notag \\
		&\leq 5\tilde{\epsilon}.\notag
	\end{align}
	Therefore we obtain
	\[
	\partial_{d_{\epsilon}} X^{\epsilon} \subseteq (\Phi(\partial_{d_{n, \epsilon}}X^{\epsilon}_{n}))_{5\tilde{\epsilon}}.
	\]
	This completes the proof.
\end{proof}

\section{Stability of quasihyperbolization}
The rest of the paper is devoted to the proof of Theorem \ref{fouth theorem}. The proof is based on the approach used already in the section on the convergence of uniformized spaces. Recall that for a metric space $(\Omega, d)$, $d(\cdot) := d(\ \cdot \ , \partial \Omega)$ and all curves are parametrized by arc-length with respect to $d$ unless otherwise stated.
\begin{prop}\label{key lemma of convergence}
Let $(\Omega, d)$ be an $A$-uniform space. Then for each $R>0$ and $p \in \Omega$,
\[
B_{k}(p, R) \subseteq \Omega\setminus(\partial \Omega)_{c}
\]
where $c:=\frac{d(p)/2}{e^{R}-1}\wedge d(p)/2$, $B_{k}(p, R)$ is a ball with radius $R$ with respect to  the quasihyperbolic metric $k$, and $(\partial \Omega)_{c}$ is the $c$-neighborhood of the boundary $\partial \Omega$ with respect to $d$.
\end{prop}
\begin{proof}
	 For each $x \in B_{k}(p, R)$, if $d(p, x)<d(p)/2$, then $x \in \Omega\setminus(\partial \Omega)_{c}$. Hence we assume that $d(p, x)\geq d(p)/2$. Take a curve $\gamma : [0, l_d(\gamma)]\to \Omega$ from $x$ to $p$ that is geodesic with respect to the quasihyperbolic metric $k$ and the closest point $y \in \partial \Omega$ from $x$.  Then we have
	\[
	d(\gamma(t)) \leq d(\gamma(t), y)\leq d(\gamma(t), x)+d(x, y)= d(\gamma(t), x)+d(x) = t+d(x).
	\]
	Hence we get
	\[
	R>\int_{0}^{l_d(\gamma)}\frac{1}{d(\gamma(t))}\, dt \geq \int_{0}^{l_d(\gamma)}\frac{1}{d(x)+t}\, dt\geq \text{log}\Big(\frac{d(x)+l_d(\gamma)}{d(x)}\Big)\geq \text{log}\Big(\frac{d(x)+d(p)/2}{d(x)}\Big).
	\] 
	 We conclude that
	\[
	e^{R} > \frac{d(x)+d(p)/2}{d(x)},
	\]
	which implies $d(p)/2 < d(x)(e^R-1).$	Therefore $x \notin (\partial \Omega)_{c}$.%, which means $x \in \Omega\setminus(\partial \Omega)_{\frac{\alpha/2}{e^{R}-1}}$
\end{proof}

\begin{lemma}\label{key lemma of quasihyperbolization 1}
Let $(\Omega, d)$ be an $A$-uniform space and $\delta'>0$ be given. For every pair of points $x, y \in \Omega$ with $d(x, y)\geq \delta'/2$, let $\gamma$ be a geodesic curve from $x$ to $y$ with respect to the quasihyperbolic metric $k$. Then we can take subcurves $(\gamma_i)_{i=1}^{N}$ such that
\[
\gamma=\sum_{i=1}^{N}\gamma_{i}, \ \ \ \delta'/2 \leq l_d(\gamma_i) < \delta', \ \ \  \text{and} \ \ \ N \leq \frac{2B\diam(\Omega)}{\delta'}.
\]
where the constant $B$ only depends on the constant $A$.
\end{lemma}
\begin{proof}
Let  $x, y \in \Omega$ and $\gamma$ be a geodesic curve between $x$ and $y$ with respect to the quasihyperbolic metric $k$. Note that, by \cite[Theorem 2.10]{BHK}, $\gamma$ is a $B:= B(A)$-uniform curve, in particular,
\begin{equation}\label{length bound quasihyperbolization}
	l_d(\gamma) \leq Bd(x, y) \leq B\text{diam}(\Omega).
\end{equation}
Split this curve $\gamma$ into $N$ subcurves $(\gamma_{i})_{i=1}^{N}$ so that $\delta'/2 \leq t_{i}:= l_d(\gamma_{i})<\delta'$ for $1\leq i \leq N$. Then we have
\[
\frac{\delta'}{2}\times N \leq \sum\limits_{i=1}^{N} l_{d}(\gamma_{i}) \leq B\text{diam}(\Omega),
\]
which tells us that $N \leq \frac{2B\text{diam}(\Omega)}{\delta'}$.
\end{proof}

\begin{rmk}\label{choice of delta' 2}
 	Let $\alpha>0$ be given. In order to show the next lemma, we note that for any $\epsilon>0$, the function $t \mapsto 1/t$ is uniformly continuous on $[\alpha, \infty)$, i.e., for any $\epsilon>0$, there exists $0<\delta' \leq \epsilon\wedge 1$ such that
 	\begin{equation}\label{uniform continuity of 1/t}
 		|\frac{1}{t}-\frac{1}{t'}| \leq \epsilon
 	\end{equation}
 	holds for every $t, t' \in [0, \infty)$ with $|t-t'| \leq 4\delta'$. 
 \end{rmk}

\begin{lemma}\label{key lemma of quasihyperbolization 2}
	Let $(\Omega, d)$ be an $A$-uniform space and $p \in \Omega$ and $R>0$ be given. Let $0<\alpha<c$ be a positive constant where $c$ is the constant given for $B_k(p, 3R)$ in Lemma~\ref{key lemma of convergence}. Then for any $\epsilon>0$, let $\delta'>0$ be a constant from Remark~\ref{choice of delta' 2} for the interval $[\alpha, \infty)$. Then for every pair of points $x, y \in B_k(p, R)$ with $d(x, y)\geq \delta'/2$, we have
	\[
	\int_{0}^{l_{d}(\gamma)}\frac{1}{d(\gamma(t))} \, dt \geq \sum\limits_{i=1}^{N}\frac{1}{d(\gamma_{i}(t_i))} l_{d}(\gamma_{i})-B\diam(\Omega)\epsilon
	\]
	where $\gamma$ is a geodesic curve from $x$ to $y$ with respect to the quasihyperbolic metric $k$, $(\gamma_i)_{i=1}^{N}$ are the subcurves of $\gamma$ taken from Lemma~\ref{key lemma of quasihyperbolization 1}, and $t_i:=l_d(\gamma_i)$.
\end{lemma}
\begin{proof}
For any $\epsilon>0$, choose $\delta'>0$ as in Remark~\ref{choice of delta' 2}. Let $x, y \in B_k(p, R)$ with $d(x, y)\geq \delta'/2$ and $\gamma \subseteq B_k(p, 3R)$ be a geodesic curve from $x$ to $y$ with respect to $k$. Then, we can take subcurves $(\gamma_i)_{i=1}^{N}$ as in Lemma~\ref{key lemma of quasihyperbolization 1}. Since $d(\gamma(t)) \geq c > \alpha$ by Lemma~\ref{key lemma of convergence}, from Remark~\ref{choice of delta' 2} we have 
	\begin{align}\label{abc}
 		\int_{0}^{l_{d}(\gamma)}\frac{1}{d(\gamma(t))} \, dt &= \sum\limits_{i=1}^{N}\int_{0}^{l_{d}(\gamma_{i})} \frac{1}{d(\gamma_{i}(t))}\, dt\notag \\
 		&\geq \sum\limits_{i=1}^{N}\int_{0}^{l_{d}(\gamma_{i})}\Big(\frac{1}{d(\gamma_{i}(t_i))}-\epsilon\Big) \, dt\notag \\
 		&\geq \sum\limits_{i=1}^{N} \frac{1}{d(\gamma_{i}(t_i))} l_{d}(\gamma_{i})-\epsilon l_{d}(\gamma)\notag \\
 		&\geq \sum\limits_{i=1}^{N}\frac{1}{d(\gamma_{i}(t_i))} l_{d}(\gamma_{i})-B\text{diam}(\Omega)\epsilon.
 	\end{align}
 	where we used \eqref{length bound quasihyperbolization} for the last inequality. This completes the proof.
\end{proof}

\begin{prop}\label{key lemma of quasihyperbolization 3}
	Let $(\Omega_n, d_n)_n \subseteq \mathcal{U}(A, R)$ and $(\Omega, d) \in \mathcal{M}$. Suppose that $(\Omega_n, d_n)_n$ are all length spaces and $d_{GHB}(\Omega_n, \Omega) \to 0$ as $n \to \infty$. Then there exist $p_n \in \Omega_n$ and $p \in \Omega$ such that for each $R>0$ and $\epsilon>0$, there exists $N \in \mathbb{N}$ such that for each $n \geq N$, we have a map $I_n : B_{k_n}(p_n, R) \to \Omega$ with $I(p_n)=p$ such that
	\[
  |k_n(x_n, y_n)-k(I_n(x_n), I_n(y_n))| \leq S\epsilon
  \]
for some $S>0$ which only depends on the constants $A>0$, $R>0$ and $p \in \Omega$.
	\end{prop}
	%\begin{rmk}
		%In order to prove Theorem \ref{p-GH of uniform spaces}, we need to construct a map satisfying three properties in Definition \ref{p-GH}. Due to that, the proof given here is very long. So I would like to give the main flow of the proof. First thing we will do is to construct a map $I_n : B_{k_n}(p_n, R) \to X$ for sufficiently large $n \in \mathbb{N}$. Secondly we will prove that this map $I_n$ satisfies the properties given in Definition \ref{p-GH}, that are
		%\begin{enumerate}
	%\item $I_n(p_{n})=p$,
	%\item $|d(I_n(x), I_n(y))-d_n(x,y)|<\epsilon$ for all $x, y \in B_{k_n}(p_{n}, R)$,
	%\item $B_{k}(p, R-\epsilon) \subseteq I_n(B_{d_n}(p_{n},R))_{\epsilon}$.
%\end{enumerate}
%$I_n(p_n)=p$ is obvious from the construction given below, so we will focus mainly on the property (2) . That is the main part of the proof, and once we are done with it, it is not difficult to see that the map $I_n$ satisfies the property (3).
	%\end{rmk}
	\begin{proof}
	 By Proposition~\ref{conv-dist-equivalent}, we may assume that $(\Omega_n)_n$ and $\Omega$ are all in a same metric space $(Z, d_{Z})$ and $d^{Z}_{H}(\Omega_n, \Omega)+d^{Z}_{H}(\partial \Omega_n, \partial \Omega) \to 0 $ as $n \to \infty$.  Pick any $p \in \Omega$ and take $p_n \in \Omega_n$ such that $d_{Z}(p_n, p) \to 0$ as $n \to \infty$. Let $R>0$ and $\epsilon>0$ be fixed. Then by Lemma~\ref{key lemma of convergence}, we know that 
		 \begin{equation}
		 	B_{k}(p, 3R) \subseteq \Omega\setminus(\partial \Omega)_{c}\notag 
		 \end{equation}	 
		 and 
		 \begin{equation}
		 	B_{k_n}(p_n, 3R) \subseteq \Omega_n \setminus(\partial \Omega_n)_{c_n} \notag 
		 \end{equation}
		 where $c=\frac{d(p)/2}{e^{3R}-1}\wedge d(p)/2$ and $c_n= \frac{d_n(p_n )/2}{e^{3R}-1}\wedge d_n(p_n)/2$.
		 Note that $c_n \to c$ as $n \to \infty$. Set $\alpha := \frac{c}{2} \wedge \frac{\tilde{c}}{2}>0$ where $\tilde{c}>0$
		 %=\frac{d(p)/2}{e^{3M}-1}\wedge d(p)/2$, $M := 4A^2\text{log}(1+\frac{20D}{17c})$ and $D := \sup_n \text{diam}(\Omega_n)$ which will appear later
		 is a constant given later. For this $\alpha>0$, let us fix the constant $\delta'>0$ in Remark~\ref{choice of delta' 2}.	 We may assume that $\delta' \leq \frac{c}{20} \wedge \frac{\tilde{c}}{20}$ for the later estimates. Then, there exists $N \in \mathbb{N}$ such that for each $n \geq N$, we have
		 \begin{equation}\label{nice choice1}
		d_{Z}(p_n, p) \leq \delta'^2, \ \ \frac{19c}{20}\leq c_n \leq \frac{21c}{20},	\ \ \text{and} \ \ d^{Z}_{H}(\Omega_n, \Omega)+d^{Z}_{H}(\partial \Omega_n, \partial \Omega)\leq \delta'^2=: \tilde{\epsilon}.
		\end{equation}
	Let $n \geq N$ be fixed. Since for each $x_n \in \Omega_n$ we can find $x \in \Omega$ such that $d_{Z}(x_n, x)\leq \tilde{\epsilon}$,  define a map $I_n : B_{k_n}(p_n, R) \to \Omega$ by $B_{k_n}(p_n, R) \ni x_n  \mapsto x \in \Omega$ with $I_n(p_n):=p$. 
	Let $x_n, y_n \in B_{k_n}(p_n, R)$ be given. We may assume that $d(x_n, y_n) \geq \delta'/2$. Take a geodesic curve $\gamma_n$ from $x_n$ to $y_n$ with respect to $k_n$. By Lemma~\ref{key lemma of quasihyperbolization 1}, we have subcurves $(\gamma_{n, i})_{i=1}^{N}$ of $\gamma_n$ such that 
	\begin{equation}\label{bound on length}
		\gamma_n=\sum_{i=1}^{N}\gamma_{n, i}, \ \ \ \delta'/2 \leq l_{d_n}(\gamma_{n, i}) < \delta' \ \ \  \text{and} \ \ \ N \leq \frac{2BD}{\delta'}.
	\end{equation}
 where $D:= \sup_n \diam(\Omega_n)$. Also note that $d_n(\gamma_{n, i}(t)) \geq c_n \geq \frac{c}{2}$ for each $t \in [0, l_{d_n}(\gamma_{n, i})]$ by the choice of $c_n$ in \eqref{nice choice1} and  $\gamma_n \subseteq B_{k_n}(p_n, 3R)$. Then, by Lemma \ref{key lemma of quasihyperbolization 2},  we have
	\begin{align}\label{abc}
 		\int_{0}^{l_{d_n}(\gamma_n)}\frac{1}{d_n(\gamma_n(t))} \, dt \geq \sum\limits_{i=1}^{N}\frac{1}{d_n(\gamma_{n, i}(t_i))} l_{d_n}(\gamma_{n, i})-BD\epsilon
 	\end{align}
 	where $t_i:=l_{d_n}(\gamma_{n, i})$. By the definition of the map $I_n : B_{k_n}(p_n, R) \to \Omega$, for $(I_n(\gamma_n(t_i)))_{i=1}^{N}$ we have
 	\begin{equation}\label{I_n estimate}
 		|d_n(\gamma_{n, i}(t_i), \gamma_{n, i+1}(t_{i+1}))- d(I_n(\gamma_{n, i}(t_i)), I_{n}(\gamma_n(t_{i+1})))|< 2\tilde{\epsilon}.
 	\end{equation}
 	Since $(\Omega, d)$ is a length space, we can pick a curve $\tilde{\gamma}_{1}$ from $I_n(x_n)$ to $I_n(\gamma_{n, 1}(t_1))$ whose length is less than $d(I_n(x_n), I_n(\gamma_n(t_1)))+\tilde{\epsilon}$ and $\tilde{\gamma}_{i}$  from $I_n(\gamma_{n, i-1}(t_{i-1}))$ to $I_n(\gamma_{n, i}(t_{i}))$ whose length is less than $d(I_n(\gamma_n(t_{i-1})), I_n(\gamma_n(t_{i})))+\tilde{\epsilon}$  for $2 \leq i \leq N$. Let $t_i':= l_{d}(\tilde{\gamma}_{i})$. Then from \eqref{I_n estimate} we get
 	\begin{equation}\label{new curve estimate}
 		t_i':= l_{d}(\tilde{\gamma}_{i})  \leq d_n(\gamma_{n, i-1}(t_{i-1}), \gamma_{n, i}(t_{i}))+3\tilde{\epsilon} \leq l_{d_n}(\gamma_{n,i})+3\tilde{\epsilon} \leq 4\delta' \notag
 		\end{equation}
 		for $1\leq i \leq N$, $t_1 \leq 4\delta'$, and 
 	\begin{equation}
 		\begin{split}
 			|d(\tilde{\gamma}_i(t_i'))-d_n(\gamma_{n, i}(t_i))| &\leq |d(I_n(\gamma_{n, i}(t_i)))-d(\gamma_{n, i}(t_i))| \notag  \\
 			&\ \ \ +|d(\gamma_{n, i}(t_i))-d_n(\gamma_{n, i}(t_i))|\notag  \\
 			&\leq 2\tilde{\epsilon} \leq 2\delta'.\notag 
 		\end{split}
 	\end{equation}
 	for $1\leq i \leq N$. Also, in order to employ \eqref{uniform continuity of 1/t}, we also note that
 	\begin{equation}\label{boundary estimate}
 		d(\tilde{\gamma}_i(t_i')) \geq d(\gamma_{n, i}(t_i))-\tilde{\epsilon} \geq d_n(\gamma_{n, i}(t_i))-2\tilde{\epsilon} \geq \frac{19c}{20}-\frac{2c}{20}=\frac{17c}{20}.
 	\end{equation}
 	Hence, from \eqref{length bound quasihyperbolization}, \eqref{uniform continuity of 1/t}, \eqref{abc}, \eqref{new curve estimate}, and the estimate
 	\[
 	\sum\limits_{i=1}^{N} \frac{1}{d_n(\gamma_{n, i}(t_i))}\tilde{\epsilon} \leq \frac{6N \tilde{\epsilon}}{c} \leq \frac{12BD}{c\delta'}\tilde{\epsilon} \leq \frac{12BD}{c}\epsilon,
 	\]
 	 we have
 		\begin{align}\label{quasi conv 1}
 		\sum\limits_{i=1}^{N}\frac{1}{d(\tilde{\gamma}_i(t_i'))}l_{d}(\tilde{\gamma}_{i}) 
 		&\leq \sum\limits_{i=1}^{N}(\frac{1}{d_n(\gamma_{n, i}(t_i))}+\epsilon)(l_{d_n}(\gamma_{n, i})+3 \tilde{\epsilon})\notag \\
 		&\leq \sum\limits_{i=1}^{N} \frac{1}{d_n(\gamma_{n, i}(t_i))} l_{d_n}(\gamma_{n, i})+3\sum\limits_{i=1}^{N} \frac{1}{d_n(\gamma_{n, i}(t_i))}\tilde{\epsilon} \notag \\
 		&\ \ +\epsilon l_{d_n}(\gamma_{n})+3\tilde{\epsilon}\epsilon \notag \\
  		&\leq \int_{0}^{l_{d_n}(\gamma_n)}\frac{1}{d_n(\gamma_n(t))} \, dt + BD\epsilon+ \frac{36BD}{c}\epsilon \notag \\
 		&\ \ + BD\epsilon+3\tilde{\epsilon}\epsilon \notag \\
 		&= \int_{0}^{l_{d_n}(\gamma_n)}\frac{1}{d_n(\gamma_n(t))}\, dt+C(B, D, c)\epsilon
 		\end{align}
 	where 
 	\[
 	C(B, D, c):= 2BD+ \frac{36BD}{c}+3.
 	\]
 	Since we have the estimate
 	\[
 	d(\tilde{\gamma}_{i}(t), \tilde{\gamma}_{i}(t_{i}')) \leq l_d(\tilde{\gamma}_{i}) \leq d_n(\gamma_{n, i}(t_i), \gamma_{n, i+1}(t_{i+1}))+3\tilde{\epsilon} \leq l_{d_n}(\gamma_{n, i+1})+3\tilde{\epsilon} \leq  4\delta, 
 	\]
 	for each $t \in [0, l_d(\tilde{\gamma}_{i})]$ ($1\leq i \leq N$), one can see, together with \eqref{boundary estimate}, that
 	\[
 		d(\tilde{\gamma}_{i}(t)) \geq  d(\tilde{\gamma}_{i}(t_{i}'))-4\delta \geq \frac{17c}{20}-\frac{4c}{100}\geq \frac{13c}{20}\geq \frac{c}{2}.
 	\]
 	Hence by \eqref{uniform continuity of 1/t} and \eqref{new curve estimate}, we have 
 	\begin{align}\label{quasi conv 2}
 		\sum\limits_{i=1}^{N}\frac{1}{d(\tilde{\gamma}_i(t_i'))}l_{d}(\tilde{\gamma}_{i}) &\geq \sum\limits_{i=1}^{N}\int_{0}^{l_d(\tilde{\gamma}_{i})}\Big(\frac{1}{d(\tilde{\gamma}_i(t))}-\epsilon \Big) \, dt \notag \\
 		&\geq \sum\limits_{i=1}^{N}\int_{0}^{l_d(\tilde{\gamma}_{i})}\frac{1}{d(\tilde{\gamma}_i(t))}\, dt -(N\epsilon)(4\delta) \notag \\
 		&\geq k(I_n(x_n), I_n(y_n))-8BD\epsilon \notag \\
 		&\geq k(I_n(x_n), I_n(y_n))-K(B, D)\epsilon
 	\end{align}
  	where $K(B, D) := 8BD$. Therefore, combining the inequalities \eqref{quasi conv 1} and \eqref{quasi conv 2}, we conclude that 
  	\[
  	k(I_n(x_n), I_n(y_n)) \leq k_n(x_n, y_n)+ K(B, D)\epsilon + C(B, D,  c)\epsilon
  	\]
  	Hence 
  	\[
  	k(I_n(x_n), I_n(y_n))- k_n(x_n, y_n) \leq K(B, D)\epsilon + C(B, D, c)\epsilon.
  	\]

  	Next, by following almost the same proof as above, we will prove that for every pair points $x_n, y_n \in B_{k_n}(p_n, R)$
  	\[
  	k_n(x_n, y_n) - k(I_n(x_n), I_n(y_n))\leq \tilde{K}(B, D)\epsilon + \tilde{C}(\tilde{c}, B, D)\epsilon
  	\]
  	 for some $\tilde{K}(B, D)$ and $\tilde{C}(\tilde{c}, B, D)$. Again pick $x_{n}$ and $y_n$ from $B_{k_n}(p_n, R)$ and a geodesic curve $\gamma$ between $I_n(x_n)$ and $I_n(y_n)$ with respect to the quasihyerbolic metric $k$. We first claim that $I_n(x_n)$ and $I_n(y_n)$ are in the large enough ball $B_k(p, M)$ for some $M>0$. Since we have the estimate
  	 \begin{equation}
  	 	d(I_n(x_n)) \geq d(x_n)-\tilde{\epsilon} \geq d_n(x_n)-2\tilde{\epsilon}\notag \geq c_n-2\tilde{\epsilon} \geq \frac{19c}{20}-\frac{2c}{20}\geq \frac{17c}{20} \notag
  	 \end{equation} 
  	  for all  $n \geq N$, by \cite[Lemma~2.13]{BHK}, we have
  	 \begin{equation}
  	 	\begin{split}
  	 		k(p, I_n(x_n)) &\leq 4A^2\text{log}(1+\frac{d(p, I_n(x_n))}{d(p) \wedge d(I_n(x_n))}) \notag \\
  	 		&\leq 4A^2\text{log}(1+\frac{20D}{17c})=:M(A, c, D)=:M.\notag 
  	 	\end{split}
  	 \end{equation}
  	 Hence a geodesic curve $\gamma$ connecting $I_n(x_n)$ and $I_n(y_n)$ stays in $B_k(p, 3M)$. By Lemma~\ref{key lemma of convergence}, we find $\tilde{c}$ such that 
  	 \[
  	 B_{k}(p, 3M) \subseteq \Omega\setminus(\partial \Omega)_{\tilde{c}}
  	 \]
  	 where $\tilde{c}=\frac{d(p)/2}{e^{3M}-1}\wedge d(p)/2$. From here by following the same argument we did above we have 
  	 \[
   k_n(x_n, y_n)-k(I_n(x_n), I_n(y_n)) \leq \tilde{K}(B, D)\epsilon + \tilde{C}(\tilde{c}, B, D)\epsilon
  	\]
where
\[
\tilde{C}(\tilde{c}, B, D):=2BD + \frac{36BD}{\tilde{c}} +3 \ \ \ \text{and} \ \ \ \tilde{K}(B, D) :=8BD.
\]
 Therefore, we conclude that
  \[
  |k_n(x_n, y_n)-k(I_n(x_n), I_n(y_n))| \leq S\epsilon
  \]
where 
\[
S := S(B, D, \tilde{c}, c):=2(\tilde{K}(B, D) \vee \tilde{C}(\tilde{c}, B, D) \vee K(B, D)\vee C(B , D, c)).
\]
\end{proof}

\begin{thm}[Theorem~\ref{fouth theorem}]\label{p-GH of uniform spaces}
	Let $(\Omega_n, d_n)_n \subseteq \mathcal{U}(A, R)$ and $(\Omega, d) \in \mathcal{M}$. Suppose $(\Omega_n, d_n)_n$ are all length spaces and $d_{GHB}(\Omega_n, \Omega) \to 0$ as $n \to \infty$. Then $(\Omega_n, p_n, k_n)_n$ is pointed Gromov-Hausdorff convergent to $(\Omega, p, k)$ for some $p_n \in \Omega_n$ and $p \in \Omega$ where $k_n$ and $k$ are quasihyperbolic metrics defined on $\Omega_n$ and $\Omega$, respectively.
	\end{thm}
	\begin{proof}
 By Proposition \ref{key lemma of quasihyperbolization 3}, for each $R>0$ and $\epsilon>0$, there exists $N \in \mathbb{N}$ such that for each $n \geq N$, we have a map $I_n : B_{k_n}(p_n, R) \to \Omega$ with $I(p_n)=p$ such that
	\[
  |k_n(x_n, y_n)-k(I_n(x_n), I_n(y_n))| \leq S\epsilon
  \]
for some $S$ which only depends on the constants $A$ and $R>0$ and a point $p \in \Omega$. It suffices to prove that $I_n$ satisfies Property ($3$) in Definition~\ref{p-GH}. Recall that by the choice of $N \in \mathbb{N}$ in the proof of Proposition~\ref{key lemma of quasihyperbolization 3}, for every $n \geq N$, we have
		 \begin{equation}
		d_{Z}(p_n, p) \leq \delta'^2, \ \ \frac{19c}{20}\leq c_n \leq \frac{21c}{20},	\ \ \text{and} \ \ 
		d^{Z}_{H}(\Omega_n, \Omega)+d^{Z}_{H}(\partial \Omega_n, \partial \Omega)\leq \delta'^2=: \tilde{\epsilon}.\notag 
	\end{equation}
	where $0<\delta' \leq \frac{c}{20} \wedge \frac{\tilde{c}}{20}$ and $c_n$, $c$, and $\tilde{c}$ are the constants in Proposition~\ref{key lemma of quasihyperbolization 3} for $\Omega_n$ and $\Omega$, respectively.
 Define a map $J : B_{k}(p, R) \to X_n$ by $x \in B_{k}(p, R) \mapsto x_n \in X_n$ such that $J(p)=p_n$ and $d_{z}(x_n, x) \leq \tilde{\epsilon}$. Then by exactly the same argument(think of $X$ as $X_n$ and $X_n$ as $X$ in the proof of Proposition~\ref{key lemma of quasihyperbolization 3}), we have 
  \begin{equation}
  	|k(p, x)-k_n(p_n, J(x))|\leq S\epsilon,\notag
  \end{equation}
  for every $x \in B_k(p, R)$. This implies that
  \begin{equation}
  	k_n(p_n, J(x)) \leq  k(p, x)+S\epsilon < R \notag
  \end{equation}
 for every $x \in B_k(p, R-S\epsilon)$. Hence $J(x) \in B_{k_n}(p_n, R)$. Also, since
 \[
 d_n(I_n(J(x)), x) \leq d_{Z}(I_n(J(x)), J(x))+ d_{Z}(J(x), x) \leq 2\tilde{\epsilon},
 \]
  by taking a curve $\gamma$ from $I_n(J(x))$ to $x$ with the length $d(I_n(J(x)), x)+ \tilde{\epsilon}\leq 3\tilde{\epsilon}$, we have
 \[
 \begin{split}
 	 k(I_n(J(x)), x) &\leq \int_{0}^{l_d(\gamma)}\frac{1}{d(\gamma(t))} \, dt\\
 	 &\leq (\frac{1}{d(x)}+\epsilon)l_d(\gamma)\\
 	 &\leq (\frac{1}{c}+\epsilon)3\tilde{\epsilon}. \notag 
 \end{split} 
 \]
 This implies that
 \begin{equation}
 	B_k(p, R-S\epsilon)\subseteq (I_n(B_{k_n}(p_n, R)))_{3(\frac{1}{c}+\epsilon)\tilde{\epsilon}}. \notag 
 \end{equation}
 Taking the maximum $S\epsilon \vee 3(\frac{1}{c}+\epsilon)\tilde{\epsilon}$, we complete the proof.
 \end{proof}
We close this paper by giving the following two corollaries. We would like to point out that by \cite[Proposition 4.37]{BHK}, it is known that a proper geodesic roughly starlike Gromov hyperbolic space $(X, d)$ is bilipschitz equivalent to the space $(X^{\epsilon}, k)$ where $X^{\epsilon}$ is the uniformized space and $k$ is the quasihyperbolic metric with respect to $d_{\epsilon}$. It is not straightforward, however, that the pointed Gromov-Hausdorff convergence $(X_n, d_n, p_n)\to (X, d, p)$ implies the pointed Gromov-Hausdorff convergence $(X_n^{\epsilon}, k_n, p_n \to (X^{\epsilon}, k, p)$.
\begin{cor}
	Let $(X_n, d_n, p_n)_n$ be a sequence of pointed proper geodesic $M$-roughly starlike $\delta$-Gromov hyperbolic spaces convergent to a pointed metric space $(X, d, p)$. Then, for a fixed $0< \epsilon \leq \epsilon_0(\delta)$, the sequence of uniformized spaces with quasihyperbolic metrics $(X_n^{\epsilon}, k_n, p_n)_n$ pointed Gromov-Hausdorff converges to $(X^{\epsilon}, k, p)$.
\end{cor}
\begin{cor}
	Let $(\Omega_n, d_n)_n$ be a sequence of bounded $A$-uniform length metric spaces convergent to a metric space $(\Omega, d)$ with respect to $d_{GHB}$. Then there exist $p_n \in \Omega_n$ and $p \in \Omega$ such that the sequence $(\Omega_n, k_{n, \epsilon})$ converges to $(\Omega, k_{\epsilon})$ with respect to $d_{GHB}$ for a fixed $0< \epsilon \leq \epsilon_0(\delta)$ where $(\Omega_n, k_{n, \epsilon})$ and $(\Omega, k_{\epsilon})$ are the uniformized spaces of $(\Omega_n, p_n, k_n)$ and $(\Omega, p, k)$, respectively.
\end{cor}

\vskip .2cm

}
\end{document}